\documentclass[11pt, twoside, leqno]{amsart}  
\makeatletter
\@namedef{subjclassname@2020}{\textup{2020} Mathematics Subject Classification}
\makeatother

\usepackage{lipsum}
\usepackage{amsfonts}
\usepackage{graphicx}
\usepackage{epstopdf}
\usepackage{calligra}
\usepackage{amsfonts,amsmath,amsthm}
\usepackage{mathtools}
\usepackage{amssymb}

\usepackage{hhline}
\usepackage{array}
\usepackage{diagbox}
\usepackage{tcolorbox}
\usepackage{mdframed}
\usepackage{multicol}
\usepackage{graphicx}
\usepackage{subcaption}
\usepackage{moreverb}
\usepackage{bbm}
\usepackage[margin=1.25in]{geometry}
\usepackage{scalerel}
\allowdisplaybreaks
\usepackage{mathrsfs}  
\usepackage{lineno}
\usepackage{todonotes}
\usepackage{tikz}
\usepackage{appendix}
\usepackage{enumitem}
\usepackage{pgfplots}
\usetikzlibrary{arrows.meta}
\usepackage{siunitx}
\usepackage[numbers,sort&compress]{natbib}
\definecolor{mygreen}{HTML}{43a047}
\usepackage{subcaption}
\usepackage{doi}
\usepackage{soul}
\usepackage[misc]{ifsym}
\usepackage{siunitx}
\usepackage{dsfont}
\usepackage{siunitx}
\definecolor{darkgreen}{rgb}{0.0, 0.2, 0.13}
\newtheorem{theorem}{Theorem}

\newtheorem{proposition}{Proposition}
\theoremstyle{definition}

\theoremstyle{remark}

\newcommand{\Om}{\Omega}

\newcommand{\betaa}{\beta_{\textup{a}}}

\newcommand{\Ga}{\Gamma_{\textup{abs}}}
\newcommand{\ut}{u_t}
\newcommand{\utt}{u_{tt}}

\newcommand{\ds}{\, \textup{d} s }
\newcommand{\dx}{\, \textup{d} \boldsymbol{x}}

\newcommand{\dG}{\, \textup{d} \Gamma}
\newcommand{\dGs}{\, \textup{d} \Gamma \textup{d}s}

\newcommand{\dxs}{\, \textup{d}\boldsymbol{x}\textup{d}s}

\newcommand{\inttO}{\int_0^t \int_{\Omega}}
\newcommand{\intt}{\int_0^t}
\newcommand{\intT}{\int_0^T}
\newcommand{\intO}{\int_{\Omega}}

\newcommand{\R}{\mathbb{R}} 
 
\newcommand{\Htwo}{H^2(\Omega)}
\newcommand{\Ltwo}{L^2(\Omega)}
\newcommand{\Linf}{L^\infty(\Omega)}

\newcommand{\LinfLinf}{L^\infty(0,T; L^\infty(\Omega))}
\newcommand{\LinfTLinf}{L^\infty(0,T; L^\infty(\Omega))}

\def\Linf{L^\infty(\Omega)}

\newcommand{\Ctr}{C_{\textup{tr}}}

\definecolor{grey}{rgb}{0.5,0.5,0.5}

\definecolor{darkgreen}{rgb}{0,0.5,0}

\def\ball{\mathbb{B}_h}

\def\calE{\mathcal{E}}

\def\uzeroh{u_{0}^h}
\def\uoneh{u_{1}^h}
\def\utwoh{u_{2}^h}

\def\Xu{\mathcal{X}_u}

\def\eps{\varepsilon}

\def\LtwoTLtwo{L^2(0,T; \Ltwo)}

\def\LinftLinf{L^\infty(0,t; \Linf)}

\def\csq{c^2}

\def\LtwoG{L^2(\Gamma)}
\def\LtwoGabs{L^2(\Gamma_{\textup{abs}})}

\def\LtwoGN{L^2(\Gamma_{\textup{N}})}

\def\uzero{u_0}
\def\uone{u_1}

\def\LtwotLtwo{L^2(0,t; \Ltwo)}

\newcommand{\Gronwall}{Gr\"onwall}

\def\lhs{\textup{lhs}}
\def\rhs{\textup{rhs}}

\newif\ifshow

\def\uttt{u_{ttt}}

\def\utwo{u_2}
\def\ah{a^{\textup{sip}}_h}

\def\uht{\uh_t}

\def\uhtt{\uh_{tt}}

\def\Ih{\mathcal{I}_h}

\def\Deltat{{\mbox{\footnotesize{$\Delta$}}} t}

\def\uhn{\uh_{n}}
\def\udothn{\dot{u}^{h}_{n}}
\def\udotdothn{\ddot{u}^{h}_{n}}

\def\uhnplusone{u^{h}_{n+1}}

\def\udothnplusone{\dot{u}^{h}_{n+1}}
\def\udotdothnplusone{\ddot{u}^{h}_{n+1}}

\def\uprev{\uh_{\textup{prev}}}
\def\udotprev{\dot{u}^h_{\textup{prev}}}

\usepackage{stmaryrd}
\newcommand{\jump}[1]{\llbracket #1 \rrbracket}
\newcommand{\avg}[1]{\{ \! \! \{ #1 \} \! \! \} }

\newcommand{\triangles}{\mathcal{T}_{h}}

\newcommand{\Ltwonorm}[2]{\|#1\|_{L^2(#2)}}

\newcommand{\dgstarseminorm}[1]{|#1|_{\textup{sip,*}}}

\newcommand{\dgseminorm}[1]{| #1 |_{\textup{sip}}}

\newcommand{\invconstant}{C_{\textup{inv}}}

\def\zhnplusone{z^h_{n+1}}
\def\zdothnplusone{\dot{z}^h_{n+1}}
\def\zdotdothnplusone{\ddot{z}^h_{n+1}}

\def\zhprev{z^h_{\textup{prev}}}
\def\zdothprev{\dot{z}^h_{\textup{prev}}}
\def\zhn{z^h_{n}}
\def\zdothn{\dot{z}^h_{n}}
\def\zdotdothn{\ddot{z}^h_{n}}

\def\calO{\mathcal{O}}

\newcommand{\asip}{a^{\textup{sip}}_{h}}

\newcommand{\ph}{p^{h}}
\newcommand{\qh}{q^{h}}

\newcommand{\uh}{u^{h}}
\newcommand{\fh}{f^{h}}

\newcommand{\fhnplusone}{f^{h}_{n+1}}

\newcommand{\ghabsnplusone}{g^{h}_{\textup{abs}, n+1}}
\newcommand{\ghNnplusone}{g^{h}_{\textup{N}, n+1}}

\newcommand{\Vh}{V^q_{h}}

\newcommand{\phih}{\phi^{h}}
\newcommand{\psih}{\psi^{h}}

\newcommand{\eI}{e^{I}}

\newcommand{\eh}{e^{h}}
\newcommand{\thh}{t^{*}_{h}}

\def\bfn{\boldsymbol{n}}

\def\Vhqp{V_h^q}

\def\qp{q}

\def\calFh{\mathcal{F}_h}
\def\calFhint{\mathcal{F}^{\textup{int}}_h}
\def\calFhbnd{\mathcal{F}^{\textup{bnd}}_h}
\def\calFhbndabs{\mathcal{F}^{\textup{abs, bnd}}_h}
\def\calFhbndN{\mathcal{F}^{\textup{N, bnd}}_h}
\def\hF{h_F}

\def\nF{\boldsymbol{n}_F}

\def\nablah{\nabla_h}

\def\intF{\int_F}

\def\Ahsip{A^{\textup{sip}}_h}

\def\Gh{G^h}

\def\Vhq{V_h^q}

\def\bfx{\boldsymbol{x}}

\def\LtwoK{L^2(K)}

\def\hK{h_K}

\def\VGh{V^q_{\Gamma, h}}

\def\VGhabs{V^q_{\Gabs, h}}
\def\VGhN{V^q_{\GN, h}}

\def\Npartial{N_{\partial}}

\def\interpolant{\mathcal{I}_h}
\def\locainterpolant{\mathcal{I}_{h}^K}

\def\ehu{e^{h}}

\def\deltahp{\delta^h}

\def\eIu{e^{I}}

\def\openset{\mathcal{M}_h}

\def\ball{\mathcal{B}}

\def\betanm{\beta_{\textup{Nm}}}
	\def\gammanm{\gamma_{\textup{Nm}}}
		\def\deltanm{\delta_{\textup{Nm}}}

\newcommand{\tnorm}[1]{\| #1 \|_{\mathcal{E}}}
\def\Czero{C_{0}}

\def\dGnorm{\textup{sip}}

\def\fhtauzero{f^{h, \tau=0}}

\def\ghabstauzero{g_{\textup{abs}}^{h, \tau=0}}
\def\ghNtauzero{g_{\textup{N}}^{h, \tau=0}}

\usepackage{etoolbox}
\makeatletter
\apptocmd{\maketitle}{%
	\@addtoreset{equation}{section}%
}{}{}
\makeatother

\def\closedball{\overline{\ball}}

\def\Lambdah{\Lambda_h}

\def\deltah{\delta^h}

\def\etanm{\eta_{\textup{Nm}}}

\def\gabs{g_{\textup{abs}}}
\def\gabstauzero{g^{\tau=0}_{\textup{abs}}}
\def\ghabs{g^h_{\textup{abs}}}
\def\gN{g_{\textup{N}}}
\def\gNtauzero{g^{\tau=0}_{\textup{N}}}
\def\ghN{g^h_{\textup{N}}}

\newtheorem{lemma}{Lemma}

\def\Gabs{\Gamma_\textup{abs}}
\def\GN{\Gamma_\textup{N}}

\def\ztt{z_{tt}}

\newenvironment{eq}
{\begin{equation}\begin{aligned}}
		{\end{aligned}\end{equation}
}

\def\btau{b_\tau}
\def\zh{z^h}

\def\ehut{ \ehu_{t}}
\def\ehutt{ \ehu_{tt}}

\def\utauzero{u^{\tau=0}}
\def\uttauzero{u_t^{\tau=0}}

\def\uhtauzero{u^{h, \tau=0}}
\def\uhttauzero{u_t^{h, \tau=0}}

\def\udiff{\bar{u}}
\def\utdiff{\bar{u}_t}
\def\uttdiff{\bar{u}_{tt}}

\def\uhttt{u^h_{ttt}}

\def\intGabs{\int_{\Gabs}}
\def\intGN{\int_{\GN}}

\newcommand{\energyseminorm}[1]{|#1|_{\textup{E}}}

\def\betanl{\beta_{\textup{nl}}}

\def\bark{\bar{k}}

\def\Linfnormk{\|k\|_{\Linf}}

\def\bartau{\bar{\tau}}
\def\ftauzero{f^{\tau=0}}   

\title[DG approximations of the JMGT equation in the vanishing relaxation limit]{Discontinuous Galerkin approximations of the Jordan--Moore--Gibson--Thompson equation in the vanishing relaxation limit}
\subjclass[2020]{65M12, 65M60}

\keywords{JMGT equation, discontinuous Galerkin methods, vanishing thermal relaxation, Newmark scheme}

\author[Vanja Nikoli\'c]{Vanja Nikoli\'c}

\address{ 
	Department of Mathematics \\ 
	Radboud University   \\ 
	Heyendaalseweg 135,
	6525 AJ Nijmegen, The Netherlands}
\email{vanja.nikolic@ru.nl} 

\begin{document}
\vspace*{8mm}
\begin{abstract}
The Jordan--Moore--Gibson--Thompson (JMGT) equation models nonlinear acoustic wave propagation in thermally relaxing media and in the vanishing relaxation limit approaches the damped Westervelt equation. We investigate discontinuous Galerkin spatial discretizations of the JMGT equation on simplicial meshes and analyze their behavior uniformly with respect to the relaxation parameter. Under practically relevant mixed Neumann and absorbing boundary conditions, we derive \emph{a priori} error estimates independent of the relaxation parameter. These estimates enable a rigorous singular limit analysis, yielding convergence of the semi-discrete JMGT approximations to the corresponding Westervelt pressure profile at a linear rate.  This also sheds light on the expected behavior of exact solutions in the vanishing relaxation limit. For the fully discrete problem, we propose a Newmark-type method based on a reformulation as a coupled second-/first-order system. Numerical experiments support the theoretical findings and demonstrate the robustness of the approach in the small-parameter regime.
\end{abstract}
\vspace*{-7mm}
\maketitle           
\vspace*{-5mm}
\section{Introduction}
\label{sec:1}
As sound waves propagate through thermally relaxing fluids, the evolution of the acoustic pressure field can be described by the Jordan--Moore--Gibson--Thompson (JMGT) equation~\cite{jordan2014second}:
\begin{equation} \label{JMGT eq} \tag{JMGT}
	\tau \uttt + ((1+k(\bfx)u)\ut)_t - \csq \Delta u - \btau \Delta \ut = f.
\end{equation}
The parameter $\tau>0$ denotes the thermal relaxation time appearing in the underlying Maxwell--Cattaneo law of heat conduction and is responsible for the third-order-in-time character of the model. In many media $\tau$ is extremely small; for example, in air it is of the order of hundreds of picoseconds~\cite{jordan2014second}. In media such as biological tissue, however, the values might be significantly higher~\cite{afrin2011thermal}. This wide range of physically relevant parameters makes it paramount to understand under which conditions solutions remain robust with respect to $\tau$. Closely related to this goal is the study of the singular perturbation limit $\tau\to0$, where, under suitable assumptions, solutions of the JMGT equation converge to solutions of the damped Westervelt equation:
\begin{equation} \label{West eq} \tag{W}
	((1+k(\bfx)\utauzero)\uttauzero)_t - \csq \Delta \utauzero - b \Delta \uttauzero = \ftauzero;
\end{equation}
see, e.g.,~\cite{kaltenbacher2020vanishing, nikolic2024nonlinear, bongarti2021vanishing, kaltenbacher2019jordan}.  The present work addresses both of these aspects at the level of numerical approximations. \\
\indent In \eqref{JMGT eq},  $c>0$ is the speed of sound and $\btau = b+\csq\tau$, where $b>0$ is the coefficient accounting for the viscosity of the medium.  The term $f=f(\bfx,t)$ represents an external acoustic source. Furthermore, $k(\bfx)=\frac{\betanl(\bfx)}{\rho \csq}$, where $\betanl=\betanl(\bfx)$ denotes the spatially varying nonlinearity parameter of the medium and $\rho$ its density. Allowing $k$ to vary spatially enables the modeling of heterogeneous media with differing strengths of nonlinearity and includes the linearized case $k=0$, which is commonly referred to as the Moore--Gibson--Thompson (MGT) equation.  In fact, the linearized equation can be traced back to the works of Stokes~\cite{stokes1851xxxviii}, Moore and Gibson~\cite{moore1960propagation}, and Thompson~\cite{thompson1972compressible}. Related linear models for acoustic propagation in inhomogeneous relaxing media were considered by
Nachman, Smith, and Waag~\cite{nachman1990equation}. Nonlinear third-order equations for relaxing media also appeared earlier in the literature in~\cite{varlamov1997third, varlamov2000long}. The nonlinear JMGT class of models to which the present equation belongs originates from the work of Jordan~\cite{jordan2014second}.

To reduce non-physical reflections from computational boundaries, we consider equation \eqref{JMGT eq} with absorbing boundary conditions of the form
\begin{equation} \label{tau abcs}
	\alpha (u+\tau \ut)_t
	+\frac{\partial}{\partial n}(\csq u+\btau \ut)
	= \gabs
	\quad \text{on } \Gabs,
\end{equation}
where $\alpha>0$. With $\gabs=0$, condition \eqref{tau abcs} can be seen as the Engquist--Majda absorbing boundary condition~\cite{abcsengquist} tailored to \eqref{JMGT eq}; we include here the boundary source $\gabs$ as it will be convenient to have it when determining empirical convergence rates in Section~\ref{sec: Newmark}. On the remainder of the boundary, we prescribe Neumann boundary excitation,
\begin{equation}
	\frac{\partial}{\partial n}(\csq u+\btau \ut)
	= \gN
	\quad \text{on } \GN=\Gamma\setminus\Gabs,
\end{equation}
and we also allow for $\GN= \emptyset$. 
Additionally, the equation \eqref{JMGT eq} is supplemented with three initial conditions,
\[
(u,\ut,\utt)|_{t=0}=(\uzero,\uone,\utwo),
\]
whereas the limiting Westervelt equation \eqref{West eq} requires only two.

\subsection*{Objectives and related literature}
The central objective of this work is to analyze numerical approximations of \eqref{JMGT eq} in the singular perturbation regime $\tau\to0$. 
For the spatial discretization, we employ discontinuous Galerkin (dG) methods on simplicial meshes. Such methods have proven effective for nonlinear wave propagation problems and exhibit excellent stability and accuracy properties for the limiting Westervelt equation \eqref{West eq}; see, e.g.,~\cite{antonietti2020high, dewit2026, gomez2025asymptotic}. For time discretization, we introduce a Newmark-type scheme for
\eqref{eq: z} whose vanishing-relaxation limit formally coincides
with the classical Newmark method for the Westervelt equation; see
\cite[Ch.~5]{kaltenbacher2007numerical}. To this end, we employ the rewriting using the combined quantity
\begin{eq}
	z= u+\tau \ut,
\end{eq}
as a second-order wave equation
\begin{eq} \label{eq: z}
	\ztt- \csq \Delta z - b \Delta \ut = \frac{k(\bfx)}{2}(u^2)_{tt} + f;
\end{eq}
see~\cite[Sec.~3.2]{kaltenbacher2026jordan} for the use of this rewriting in the continuous analysis. The formulations \eqref{JMGT eq} and \eqref{eq: z} are equivalent
at the semi-discrete level; however, the form \eqref{eq: z}
provides a convenient starting point for the construction of the
time discretization. \\
\indent Existing numerical analysis results for the JMGT equation are limited. Error estimates for implicit Runge--Kutta discretizations under homogeneous Dirichlet boundary conditions were derived in~\cite[Theorem 4.2]{kaltenbacher2021convergence}, while conforming finite element approximations of the linear Moore--Gibson--Thompson equation were analyzed in~\cite{knapp2014stability}.
Although discontinuous Galerkin methods are very well established, see, for example, the books~\cite{di2011mathematical, riviere2008discontinuous, cohen2017finite, CangianDongGeorgoulisHouston2017}, the review paper~\cite{antonietti2022mathematical}, and the references therein, their application to the JMGT equation has not yet been investigated from the perspective of the vanishing relaxation time analysis. \\
\indent To the best of our knowledge, this is the first work to establish $\tau$-uniform energy estimates for a discontinuous Galerkin discretization of the JMGT equation and to rigorously analyze the convergence of the resulting semi-discrete approximations toward the Westervelt limit.

For the continuous problem, the vanishing relaxation time limit has been investigated in~\cite{kaltenbacher2019jordan, bongarti2021vanishing, kaltenbacher2020vanishing, nikolic2024nonlinear}. In particular,~\cite[Theorem 5.1]{nikolic2024nonlinear} establishes, as a particular case, strong convergence of solutions $u$ of \eqref{JMGT eq} to $\utauzero$ at a linear rate in $\tau$ in a suitable norm under Dirichlet boundary conditions. More generally, the literature on the well-posedness and qualitative behavior of the JMGT equation and its linearized MGT counterpart is extensive; see, for example,~\cite{kaltenbacher2020vanishing, kaltenbacher2012well, kaltenbacher2011wellposedness}, the review~\cite{kaltenbacher2026jordan}, and the references therein. \\
\indent A central difficulty in the vanishing-relaxation analysis is obtaining
estimates that uniformly prevent the factor $1+ku$ from
degenerating. Such control is required to ensure that
\eqref{West eq} remains a valid nonlinear wave model and is
consistent with the weakly nonlinear acoustic regime underlying
its derivation~\cite{hamilton1998nonlinear}. Establishing analogous uniform control is one of the key challenges in the semi-discrete setting and plays a central role in the convergence analysis carried out below.
\subsection*{Main findings} In this work, we establish the following results, all with the practically relevant mixed Neumann--absorbing boundary conditions \eqref{tau abcs}:
\begin{itemize}[leftmargin=*, itemsep=2pt] 
	\item[$\circ$] we prove optimal \emph{a priori} energy estimates for the dG semi-discretization of \eqref{JMGT eq} that are uniform with respect to the relaxation parameter $\tau$;
	
	\item[$\circ$] we prove that the semi-discrete dG approximations converge to the corresponding Westervelt solution with order $\calO(\tau)$ in a suitable norm;
	
	\item[$\circ$] we develop a Newmark-type fully discrete scheme for \eqref{eq: z}, tailored to the small-relaxation regime, and investigate its performance numerically.
\end{itemize}
The present analysis also provides new insight into the asymptotic behavior of the continuous problem in the singular limit. For mixed Neumann--absorbing boundary conditions, only weak convergence results appear to be available so far in~\cite[Theorem 6.6]{kaltenbacher2020vanishing}. In contrast to the present work, the JMGT model considered there is formulated in the so-called potential form, with nonlinearity $k(u_t^2)_t$, and is equipped with $\tau$-independent absorbing boundary conditions (i.e., \eqref{tau abcs} with $\tau=0$) and homogeneous initial data.

All of our results also apply to the linear MGT equation ($k=0$), although we do not pursue optimal assumptions on the polynomial degree or solution regularity in that setting. 

The proof of the $\tau$-uniform energy estimates follows the general framework of \cite{hochbruck22, maier2020error}; see also \cite{dorich2025strong, careaga2026finite, dewit2026}. We first establish well-posedness and error estimates on a possibly discretization-dependent time interval and subsequently extend these results to $[0,T]$ by deriving energy bounds that are uniform in both $\tau$ and the discretization parameters.

The main technical challenge lies in constructing semi-discrete energy analysis. While our approach is inspired by the continuous analyses in \cite{kaltenbacher2020vanishing, nikolic2024nonlinear}, those arguments cannot be transferred directly to the dG setting. In particular, the continuous theory relies on global smoothness properties of the solution, such as $u(t)\in H^{3/2+s}(\Omega)$ with $s \in [0, 1/2)$, which are unavailable for discontinuous finite element approximations. 
\subsection*{Organization} The remainder of the paper is organized as follows.  Section~\ref{Sec: main results} contains the theoretical preliminaries as well as the statements of the main theoretical results: $\tau$-uniform \emph{a priori} bounds in Theorem~\ref{thm: main} and the $\tau$ limiting behavior of semi-discrete solutions in Theorem~\ref{thm: tau limit}.  We further devise a Newmark time discretizations strategy and numerically investigate the problem.  Section~\ref{Sec: proof} is dedicated to the proof of optimal \emph{a priori} bounds in Theorem~\ref{thm: main}. The final Section~\ref{Sec: proof 2} contains the proof of the vanishing relaxation convergence.  \\[2mm]
\noindent \textbf{Notation}. Throughout the paper, we use $\lhs \lesssim \rhs$ to denote $\lhs \leq C \cdot \rhs$ for a constant $C>0$ independent of both the mesh size $h$ and the thermal relaxation parameter $\tau$.  Given a Hilbert space $H$,  its inner product is denoted by  $(\cdot, \cdot)_{H}$.\\
\section{Main results:  \emph{a priori} estimates and \texorpdfstring{$\tau$}{tau}  convergence}
\label{Sec: main results}

In this section, we state and numerically illustrate the two main theoretical results of this work,
uniform \emph{a priori} estimates and convergence in the relaxation limit
$\tau \to 0$ for the initial boundary-value problem associated with \eqref{JMGT eq}:
\begin{equation}\label{exact jmgt problem} \tag{P$^\tau$}
	\begin{cases}
		\tau \uttt + ((1+k(\bfx)u)\ut)_t -\csq \Delta u - \btau \Delta \ut = f \qquad & \textup{ in }  \Omega \times (0,T) ,\\[1mm]
		\alpha	(u+\tau \ut)_t+ 	\frac{\partial }{\partial n}  (\csq u+\btau \ut) = \gabs& \textup{ on } \Ga \times (0,T),\\[2mm]
		\frac{\partial }{\partial n}  (\csq u+\btau \ut) = \gN& \textup{ on } \GN \times (0,T),\\[2mm]
		(u, \ut, \utt) = (\uzero, \uone, \utwo) & \textup{ at } \Omega  \times \{0\}.
	\end{cases}
\end{equation}
We assume that the given initial data and source functions are such that the exact problem \eqref{exact jmgt problem} is well posed and admits a sufficiently regular solution satisfying
\begin{eq}\label{def Xu}
	u \in \Xu
	:= W^{2,\infty}(0,T;H^{q+1}(\Omega))
	\cap H^{3}(0,T;H^q(\Omega)),
	\qquad q \geq d/2,
\end{eq}
and that there exists $r>0$ such that
\begin{eq} \label{non-degeneracy}
	\|ku\|_{C([0,T];L^\infty(\Omega))} < r < 1.
\end{eq}
The smallness condition \eqref{non-degeneracy} implies that $1+ku > 1-r >0
\quad \text{a.e. in } \Omega\times(0,T)$, and therefore guarantees that the leading term in the limiting Westervelt equation \eqref{West eq} remains non-degenerate. In the well-posedness theory of nonlinear acoustic models, the non-degeneracy condition \eqref{non-degeneracy} is typically ensured by imposing suitable smallness assumptions on the source, initial, and boundary data; see, for example,~\cite[Theorem 6.5]{kaltenbacher2020vanishing}.  Here, we instead assume the nonlinearity coefficient $k$ to be sufficiently small in $L^\infty(\Omega)$. This allows us to avoid a smallness assumption on the discretization parameter and to include the polynomial degree case $q = d/2$ in the analysis. 

\subsection{Discrete spaces and notation}
We discretize \eqref{exact jmgt problem} in space on a quasi-uniform, shape-regular 
simplicial mesh $\triangles$ that exactly partitions $\Omega$. Let $h_K = \textup{diam}(K)$, 
$h = \max_K h_K$, and let $\calFh = \calFhint \cup \calFhbnd$ denote all faces, split 
into interior and boundary faces, with $\calFhbndN$, $\calFhbndabs$ the subsets lying 
on $\GN$ and $\Gabs$. The local face parameter $\hF$ is the diameter of $F$ for $d \geq 2$. To make sense of the dG formulation when $d=1$, we set $\hF = \min(h_{K_1}, h_{K_2})$ for interior faces in that case and $\hF = h_K$ for boundary faces.

For an interior face $F = \partial K_1 \cap \partial K_2$ with fixed normal 
$\bfn_F = \bfn_{K_1} = -\bfn_{K_2}$, the jump and average operators are
\begin{equation}
	\jump{\phi} = \phi_{\vert K_1} - \phi_{\vert K_2}, \qquad 
	\avg{\phi} = \tfrac{1}{2}(\phi_{\vert K_1}+\phi_{\vert K_2}),
\end{equation}
with $\jump{\phi} = \avg{\phi} = \phi$ on boundary faces.  We denote by $\mathbb{P}^{q}(K)$ the space of polynomials of degree $q \geq 1$ on $K \in \triangles$. The space of approximate solutions is defined by
\begin{eq}
	\Vh = \bigl\{\phih \in \Ltwo \mid \phih_{\vert K} \in \mathbb{P}^{q}(K),\ 
	\forall K \in \triangles \bigr\}, \quad q \geq 1,
\end{eq}
equipped with the semi-norm
\begin{eq}
	\dgseminorm{\phih}^2 = \sum_{K \in \triangles} \|\nabla \phih\|_{L^{2}(K)}^{2} 
	+ \sum_{F \in \calFhint} \hF^{-1}\|\jump{\phih}\|_{L^{2}(F)}^{2}.
\end{eq}
The broken Sobolev space $H^{m}(\triangles) = \{\phi \in L^{2}(\Omega) \mid 
\phi_{\vert K} \in H^{m}(K),\ \forall K \in \triangles\}$, $m \geq 1$, carries the norm
\begin{eq}
	\|\phi\|_{H^{m}(\triangles)} = \Bigl(\sum_{K \in \triangles} 
	\|\phi\|^2_{H^{m}(K)}\Bigr)^{1/2},
\end{eq}
and the broken gradient operator $\nablah: H^1(\triangles) \rightarrow [\Ltwo]^d$ is defined element-wise by $(\nablah v)_{\vert K} = \nabla(v_{\vert K})$;  see~\cite[Def.~1.2.1]{di2011mathematical}.

\subsection{Semi-discretization}  To discretize the problem, we employ the symmetric interior penalty (SIP) Galerkin bilinear form $\asip: \left(H^2(\Omega) + \Vhq\right) \times \Vhq \rightarrow \R$ given by 
\begin{eq}  \label{def: ah}
 \begin{multlined}[t]
 		\asip(\psi, \phih) 
 	=  \intO  \nablah \psi \cdot \nablah \phih \dx 		+ \sum_{F \in \calFhint} \int_{F} \frac{\chi}{\hF} \jump{\psi} \jump{\phih} \dG
	\\	\hspace*{2cm}- \sum_{F \in \calFhint} \intF \left( \jump{\psi}\avg{\nablah \phih }\cdot \nF + \jump{\phih}\avg{ \nablah \psi }\cdot \nF \right)  \dG, \vspace*{-2mm}
	\end{multlined}
\end{eq}
where $\chi>0$ is the stabilization parameter. The summations in the consistency, symmetry, and penalty terms in $\asip$ are limited to interior faces; cf.~\cite[eq.~(4.16)]{di2011mathematical}. The absorbing boundary condition will be enforced by incorporating the term $\alpha( \uht+ \tau \uhtt, \phih)_{\LtwoGabs}$  into the formulation~\eqref{semi-discrete jmgt},  while the Neumann data will enter through the linear functional  on the right-hand side.

The semi-discretization of \eqref{exact jmgt problem} then reads as follows: We search for $\uh  \in C^3([0,T]; V_h)$, which, for all $\phih \in \Vhq$ and $t \in (0,T]$ satisfies
\begin{equation} \label{semi-discrete jmgt} \tag{$P^\tau_h$}
	\left\{	\begin{aligned}
		&\begin{multlined}[t](\tau \uhttt, \phih)_{\Ltwo}+(((1+k \uh)\uht)_t, \phih)_{\Ltwo}+\ah(\csq\uh+\btau \uht, \phih) \\
			+ \alpha( \uht+ \tau \uhtt, \phih)_{\LtwoGabs}
		\end{multlined}
		\\
		=&\, (\fh, \phih)_{\Ltwo} + (\ghabs, \phih)_{\LtwoGabs}  + (\ghN, \phih)_{\LtwoGN}, \\[1mm]
		& (\uh, \uht, \uhtt)\vert_{t=0}=(\uzeroh,  \uoneh, \utwoh) \in \left(\Vh\right)^3.
	\end{aligned} \right.
\end{equation}
The functions  $\fh \in  C([0,T]; \Vh)$, $\ghabs \in H^1(0,T; \VGhabs)$, $\ghN \in H^1(0,T; \VGhN)$  are assumed to approximate $f$, $\gabs$, and $\gN$, respectively, in the following sense:
\begin{eq} \label{approx properties source terms}\
	& \|f - \fh\|_{L^{2}(0,T;L^{2}(\Omega))} \lesssim h^{q},  \quad &&
	\|\gabs- \ghabs\|_{H^1(0,T;L^{2}(\Gabs))} \lesssim h^{q},  \\
	& \|\gN- \ghN\|_{H^1(0,T;L^{2}(\GN))} \lesssim h^{q+1}, &&
\end{eq}
where $\VGhabs$ and $\VGhN$ denotes the trace space of $\Vh$ on $\Gabs$ and $\GN$.  

The higher order of approximation for $\ghN$ is required in the convergence analysis to estimate Neumann boundary contributions in the error estimates as these involve an application of an inverse trace inequality; see \eqref{est ghN} in the proof of Lemma~\ref{lemma: defect}. The assumed approximation properties are satisfied, for instance, by choosing
$f_h$, $\gabs$, and $\gN$ as suitable interpolants or $L^2$-projections of the exact data onto the corresponding discrete spaces, provided the latter are sufficiently regular. 
\subsection{Uniform \emph{a priori} bounds} We now state the first main theoretical result which concerns \emph{a priori} error bounds measured in the following energy norm:
\begin{eq} \label{def calE}
	\tnorm{u(t)}^2 
	=\,	\begin{multlined}[t]
		\frac{\tau}{2}\| \utt(t)\|^2_{\Ltwo}+ \intt  \|\utt(s)\|^2_{\Ltwo}\ds 	+ \|\ut(t)\|^2_{\Ltwo} 	
		+| \ut(t) |^2_{\dGnorm}\\+ \| \ut(t)\|^2_{\LtwoGabs}+  \tau \intt \| \utt (s)\|^2_{\LtwoGabs}\ds
+ \|u(t)\|^2_{\Ltwo} \\+| u(t) |^2_{\dGnorm} .
	\end{multlined}
\end{eq}
This norm is the natural one for the semi-discretization of \eqref{JMGT eq} after testing with $\uhtt$,  which will be the test function of choice in the energy analysis.

The approximate initial data are assumed to interpolate the exact ones, where the interpolant operator is given by
\begin{equation} \label{def interpolant}
	\left(	\interpolant \phi \right)_{ \vert K} = \locainterpolant \phi, \quad K \in \triangles,
\end{equation}
with $\locainterpolant \phi$ being the local continuous interpolant; we refer to Lemma~\ref{lem:  local interpolant properties} for its properties. As we are interested in the small-parameter regime, we assume without loss of generality that $\tau \in (0, \bar{\tau}]$, for some fixed $\bar{\tau}>0$. \vspace*{-2mm}
\begin{theorem}[$\tau$-uniform \emph{a priori} error bounds] \label{thm: main}
Let $\tau \in (0, \bar{\tau}]$ for some fixed $\bar{\tau}>0$, and let $\csq$, $ b> 0$. Let the assumptions on $\triangles$ made this section hold and let the polynomial degree be $\qp \geq d/2$, where $d \in \{1,2,3\}$. 
Suppose $u \in \Xu$ is the solution of the exact JMGT problem that satisfies the non-degeneracy condition \eqref{non-degeneracy}, where the space $\Xu$ is defined in \eqref{def Xu}.
	
	Let approximate data $\fh$, $\ghabs$, and $\ghN$ satisfy the accuracy assumptions made in \eqref{approx properties source terms} and let approximate acoustic initial conditions be chosen to interpolate the exact ones:
	\begin{eq}
		(\uzeroh,  \uoneh, \utwoh) = (\Ih \uzero, \Ih \uone, \Ih \utwo). 
	\end{eq}
	Then there exists $\bark>0$, such that for all $\|k\|_{\Linf} \leq \bark$,  problem \eqref{semi-discrete jmgt} has a unique  solution $\uh \in C^3([0,T]; \Vhqp)$ satisfying
	\begin{eq} \label{error bound pressure tnorm}
		\tnorm{{u(t)-\uh(t)}}^2	\lesssim h^{2q}
	\end{eq}
	for $t \in [0,T]$.  The hidden constants depend on $\|u\|_{\Xu}$ and $T$, but not on the discretization parameter $h$ nor on the thermal relaxation time $\tau$. 
\end{theorem}

For the discussion of the time discretization scheme, we are also interested in the accuracy of the combined variable $z=u+\tau u_t$. We can deduce this from Theorem~\ref{thm: main}.  To this end, we introduce the semi-discrete dG version of the usual wave energy:
\begin{eq} \label{energy seminorm}
	\energyseminorm{\uh(t)} = \left(\|\uht(t)\|^2_{\Ltwo} + \dgseminorm{\uh(t)}^2\right)^{1/2}, \quad \uh(t) \in \Htwo+\Vh.
\end{eq}
Then with $\zh=\uh+\tau \uht$, the energy semi-norms satisfy
\begin{equation}
	\energyseminorm{u(t)-\uh(t)} + \energyseminorm{z(t)-\zh(t)}
	\lesssim
	\energyseminorm{u(t)-\uh(t)}
	+ \tau\,\energyseminorm{u_t(t)-\uht(t)}
\end{equation}
for $t\in[0,T]$. Using \eqref{energy seminorm} and the boundedness of $\tau \in (0,\bar\tau]$, this yields
\begin{eq} \label{conv rate z}
&	\energyseminorm{u(t)-\uh(t)} + \energyseminorm{z(t)-\zh(t)}\\
	\lesssim  &\,
	\|\ut-\uht\|_{\Ltwo}
	+ \dgseminorm{u-\uh}
	+ \sqrt{\tau}\|\utt-\uhtt\|_{\Ltwo} 
		\lesssim \, h^{2q}.
\end{eq}
We postpone the proof of Theorem~\ref{thm: main} to Section~\ref{Sec: proof} and discuss next the asymptotic behavior of semi-discrete solutions in the vanishing relaxation limit.
\subsection{Vanishing relaxation limit} The second main theoretical result concerns the connection between $\uh$ and $\uhtauzero$, where the latter solves the limiting semi-discrete Westervelt problem: 
\begin{equation} \label{semi-discrete west} \tag{$P^{\tau = 0}_h$}
	\left\{	\begin{aligned}
		&(((1+k \uhtauzero)\uhttauzero)_t, \phih)_{\Ltwo}+\csq\ah(\uhtauzero, \phih)
\\& \hspace*{5cm}		+ \alpha( \uhttauzero, \phih)_{\LtwoGabs}\\
		=&\, (\fhtauzero, \phih)_{\Ltwo} + (\ghabstauzero, \phih)_{\LtwoGabs}+ (\ghN, \phih)_{\LtwoGN} \\ &\text{  for all    } \phih \in \Vh, \ t \in (0,T],  \\
		& (\uhtauzero, \uhttauzero)\vert_{t=0}=(\uzeroh,  \uoneh).
	\end{aligned} \right.
\end{equation}  
In practical ultrasound applications, the excitation data are typically independent of the thermal relaxation parameter, since $\tau$ characterizes the propagation medium. Nevertheless, to cover the manufactured setting used in the numerical experiment of Section~\ref{Sec: numerics known solution}, we allow the discrete source terms $\fh$ and $\ghabs$ to depend here on $\tau$.  The following statement establishes strong linear convergence of the solutions of \eqref{semi-discrete jmgt}  as $\tau \rightarrow 0$ to the solution of \eqref{semi-discrete west}. \vspace*{-3mm}
\begin{theorem}[$\tau$ convergence] \label{thm: tau limit}
	Let the assumption of Theorem~\ref{thm: main} hold and let
	\begin{eq} \label{tau conv assumption fh ghabs}
		&\|\fh-\fhtauzero\|_{L^2(0,T; \Ltwo)} \lesssim \tau, \quad 	\|\ghabs-\ghabstauzero\|_{L^2(0,T; \LtwoGabs)} \lesssim \tau, 
	\end{eq}
	and assume that $\ghN$ is $\tau$-independent, so that $\ghN=\ghNtauzero $. Then the family $\{\uh\}_{\tau >0}$ of solutions to \eqref{semi-discrete jmgt} converges to the solution $\uhtauzero$ of the semi-discrete Westervelt problem \eqref{semi-discrete west} as $\tau \rightarrow 0$ at a linear rate:
	\begin{eq} \label{tau bound pressure}
		\begin{multlined}[t]
			\max_{t \in [0,T]}| \uh(t)-\uhtauzero(t) |^2_{\Ltwo}+ \int_0^T \dgseminorm{\uh(s)-\uhtauzero(s)}^2\ds 
			\lesssim \tau^2.
		\end{multlined}
	\end{eq}
\end{theorem}
Previously, only weak convergence has been established for the continuous problem with mixed Neumann–Enquist--Majda boundary conditions in~\cite[Theorem 6.6]{kaltenbacher2020vanishing}, where a different nonlinearity of the form $ (u^2_t)_t$ is considered and additional restrictions are imposed, including constant $k$ and homogeneous initial data. Theorem~\ref{thm: tau limit} suggests that a corresponding strong convergence result with a linear rate may also hold at the continuous level in the present setting. For the continuous JMGT equation with Dirichlet boundary conditions, a linear convergence rate was established in~\cite[Theorem 5.1]{nikolic2024nonlinear}. 

We postpone the proof of Theorem~\ref{thm: tau limit} to Section~\ref{Sec: proof 2} to first numerically investigate the problem.

\subsection{Full discretization}  \label{sec: Newmark}
For the time discretization, as announced, we employ the $z$ rewriting of the problem:
\begin{equation} \label{z u form}
 \begin{aligned}
			 z_{tt} - c^2 \Delta z - b \Delta u_t = -\frac{k}{2} (u^2)_{tt} +f, \qquad
		z= \tau \ut+ u.
	\end{aligned} 
\end{equation}
The Newmark time discretization scheme is known to display good properties for longer simulation times in realistic parameter settings 	for the limiting Westervelt equation; see, for example,~\cite[Chapter 5]{kaltenbacher2007numerical}.	We thus employ a time-discretization scheme for \eqref{z u form} that formally corresponds to the Newmark method for the Westervelt equation in the limit $\tau \rightarrow 0$. 

We first discretize the wave equation for $z$ using the Newmark method with a semi-implicit approach for nonlinearities. Let $t_0=0 < t_1 < \ldots < t_{N-1} < t_N=T$ be a uniform discretization of $[0,T]$ with the time step $\Deltat = t_{n}-t_{n-1}$, $n \geq 1$.  Given parameters $\betanm, \gammanm \in [0,1]$, we seek
$(\zhnplusone, \zdothnplusone, \zdotdothnplusone) \in (\Vh)^3$ satisfying 
\begin{equation} \label{Newmark relations z}
	\begin{aligned}
		\zhnplusone &= \zhn + \Deltat\,\zdothn
		+ \frac{\Deltat^2}{2}\Bigl(
		(1-2\betanm)\,\zdotdothn + 2\betanm\,\zdotdothnplusone
		\Bigr), \\
		\zdothnplusone &= \zdothn + \Deltat\Bigl(
		(1-\gammanm)\,\zdotdothn + \gammanm\,\zdotdothnplusone
		\Bigr) 
	\end{aligned}
\end{equation}
for $n \geq 0$, with
\begin{eq}	\label{eq:weak-z}
	\bigl(\zdotdothnplusone,\, \phih\bigr)_{\Ltwo}
	+ \csq\,\ah\bigl(\zhnplusone,\, \phih\bigr)
	+ b\,\ah\bigl(\udothnplusone,\, \phih\bigr)
	+ \alpha\,\bigl(\zdothnplusone,\, \phih\bigr)_{\LtwoGabs} \\
	= \begin{multlined}[t] - \bigl(k
	\uhnplusone\,\udotdothnplusone + k(\udothnplusone)^2,\,
	\phih
	\bigr)_{\Ltwo}
	+ \bigl(\fhnplusone,\, \phih\bigr)_{\Ltwo}
\\	+ \bigl(\ghabsnplusone,\, \phih\bigr)_{\LtwoGabs}+ \bigl(\ghNnplusone,\, \phih\bigr)_{\LtwoGN}.
	\end{multlined}
\end{eq}
Here $\fhnplusone = \fh(\cdot, t_{n+1})$, $\ghabsnplusone= \ghabs(\cdot, t_{n+1})$, and $\ghNnplusone= \ghN(\cdot, t_{n+1})$.	

If $b=k=0$ in \eqref{eq: z}, the error and stability analysis of \eqref{Newmark relations z}, 	\eqref{eq:weak-z} can be adapted from~\cite[Chapter 8.6]{raviart1983introduction}, where $H^1$ conforming finite element discretization in space is employed, and the problem is equipped with homogeneous Dirichlet boundary data. The scheme is then unconditionally stable and second-order accurate if $\gammanm=\frac12$ and $\betanm=\frac14$.

To discretize the $z$--$u$ relation, we again employ a Newmark approach. Given parameters $\deltanm, \etanm \in [0,1]$, we seek $(\uhnplusone, \udothnplusone) \in (\Vh)^2$ satisfying for $n \geq 0$
\begin{eq}\label{Newmark relations uh}
	\uhnplusone =&\, \uhn+ \Deltat \udothn+ \frac{\Deltat^2}{2} ((1-2\deltanm)\udotdothn+ 2\deltanm \udotdothnplusone), \\
	\udothnplusone =&\, \udothn+ \Deltat ((1-\etanm)\udotdothn+ \etanm \udotdothnplusone).
\end{eq}
The Newmark scheme is initialized with $\uh_0 = \Ih \uzero$, $\dot{u}^h_0 = \Ih \uone$, 
$\ddot{u}^h_0 = \Ih \utwo$, and $\zh_0 = \Ih(\uzero + \tau \uone)$,  $\dot{z}^h_0 = \Ih(\uone + \tau \utwo)$. \vspace*{2mm}

\noindent \textbf{Rewriting \eqref{eq:weak-z} in terms of one unknown}.   We introduce the following quantities on the previous time level $n$:
\begin{eq}
	\zhprev =&\, \zhn + \Deltat\,\zdothn
	+ \frac{\Deltat^2}{2}
	(1-2\betanm)\,\zdotdothn \quad\quad && \zdothprev=\, \zdothn + \Deltat 
	(1-\gammanm)\,\zdotdothn, \\
	\uprev=&\,  \uhn+ \Deltat \udothn+ \frac{\Deltat^2}{2} (1-2\deltanm)\udotdothn, \quad && \udotprev=\,  \udothn+ \Deltat (1-\etanm)\udotdothn.
\end{eq}
Then the Newmark relations \eqref{Newmark relations z}  and \eqref{Newmark relations uh} for $z^h$ and $\uh$ can be written as
\begin{equation} 
	\begin{aligned}
		\zhnplusone = \zhprev+ \betanm \Deltat^2 \zdotdothnplusone, \qquad
		\zdothnplusone = \zdothprev+ \gammanm \Deltat \zdotdothnplusone,
	\end{aligned}
\end{equation}
and
\begin{equation} 
	\begin{aligned}
		\uhnplusone = \uprev+ \deltanm \Deltat^2 \udotdothnplusone, \qquad
		\udothnplusone = \udotprev+ \etanm \Deltat \udotdothnplusone.
	\end{aligned}
\end{equation}
The relation $\tau \ut+u =z$ is discretized as follows:
\begin{eq}
	\tau ( \udotprev + \Deltat \etanm \udotdothnplusone ) + \uprev + \Deltat^2 \deltanm \udotdothnplusone = \zhnplusone
\end{eq}
from which we infer
\begin{eq}  \label{def uhn dotdot via zhn} 
	\udotdothnplusone 
	=&\, D_n + \zeta_\tau \zdotdothnplusone,
\end{eq}
where we have introduced the short-hand notation
\begin{eq}
	D_n =&\, \frac{\zhprev-\uprev-\tau \udotprev}{\Deltat(\deltanm \Deltat + \tau \etanm)}, \quad \zeta_\tau = \frac{\betanm \Deltat}{\deltanm \Deltat + \tau \etanm}
\end{eq}
Further, 
\begin{eq}  \label{def uhn derivatives via zhn} 
	\udothnplusone = \udotprev+ \Deltat \etanm (D_n+ \zeta_\tau \zdotdothnplusone) = C_n + \lambda_\tau \zdotdothnplusone, 
\end{eq}
where
$	C_n = \udotprev + \Deltat \etanm D_n $ and $\lambda_\tau = \Deltat \etanm \zeta_\tau$.
We substitute  \eqref{def uhn dotdot via zhn} and \eqref{def uhn derivatives via zhn} into \eqref{eq:weak-z} together with \eqref{Newmark relations z} to obtain
\begin{eq}	\label{eq:assembled}
	&\begin{multlined}[t]	\bigl(\zdotdothnplusone,\,\phih\bigr)_{\Ltwo}+\csq\, \betanm \Deltat^2 \ah(\zdotdothnplusone, \phih)+  b \lambda_\tau\, \ah(\zdotdothnplusone, \phih) \\[1mm]
		+ \alpha \gammanm\,\Deltat\, \bigl( \zdotdothnplusone,\,\phih
		\bigr)_{\LtwoGabs} 
		+ \, \bigl(k
		(	\uhn\,\zeta_\tau 
		+ \udothn\,\lambda_\tau )\zdotdothnplusone,\,
		\phih
		\bigr)_{\Ltwo} 
	\end{multlined}
	\\
	=&\,\begin{multlined}[t]  \bigl(\fhnplusone,\,\phih\bigr)_{\Ltwo}
		+ \bigl(\ghabsnplusone,\,\phih\bigr)_{\LtwoGabs}+ \bigl(\ghNnplusone,\,\phih\bigr)_{\LtwoGN}\\[1mm]
		- \csq\,\ah\bigl(\zhprev,\,\phih\bigr)-b \,\ah\bigl(C_n, \phih\bigr)
		- \alpha( \zdothprev,\phih)_{\LtwoGabs}	\\-  \bigl(k
		\uhn\,D_n
		+ \udothn\,C_n,	\phih	\bigr)_{\Ltwo}
	\end{multlined}
\end{eq}
for all $\phih \in \Vh$. This formulation corresponds to the so-called effective mass matrix formulation of the Newmark scheme. It is solved for $\zdotdothnplusone$, after which we employ  \eqref{def uhn dotdot via zhn} to obtain  $\udotdothnplusone$, and then \eqref{Newmark relations uh}  to recover $\uhnplusone$.  

To (formally) recover the Newmark scheme for the Westervelt equation in the limit $\tau \rightarrow 0$ where we expect $\zhnplusone$ and $\uhnplusone$ to be close, we set $\betanm = \deltanm$ and $\gammanm = \etanm$, so that $\zeta_\tau \rightarrow 1$ and $\lambda_\tau \rightarrow \gammanm \Deltat$ as $\tau \rightarrow 0$ for any fixed $\Deltat > 0$.
	
	\subsection{Numerical experiments}
	We next conduct numerical experiments in two computational settings: one academic example with a manufactured solution and one where practically relevant focusing of ultrasound waves can be observed. All simulations are performed using the open-source finite element library FEniCSx 0.9.0~\cite{fenicsx}.\footnote{All source codes are available at \href{https://github.com/nikolicvanja/dg_jmgt.}{https://github.com/nikolicvanja/dg\_jmgt}
}
	\subsubsection{Known solution} \label{Sec: numerics known solution} The first set of experiments is performed on the domain $\Omega= (0,1/4) \times (0,1/4)$ until final time $T=0.5$. The (non-dimensionalized) acoustic parameters are chosen as
	\begin{eq}
		c =\, 1, \quad b = 0.01, \quad k = 0.7, \quad \alpha = c
	\end{eq}
	and we consider two different relaxation regimes: $\tau =10^{-6}$ and $\tau =1$. 
	We set $\GN= \emptyset$, so that we have absorbing-type conditions on the whole boundary. 
	The source functions $f$ and $\gabs$ and initial data $(\uzero, \uone, \utwo)$ are chosen so that the function
	\begin{eq}
		u(x,y,t) =  \sin(\pi x/4) \sin(\pi y/4) \cos(t).
	\end{eq} 
	is the exact solution of the problem.
	The experiments are conducted with Newmark parameters $(\betanm, \gammanm) = (\frac14, \frac12)$ and $(\deltanm, \etanm) =(\betanm, \gammanm)$. The dG stabilization parameter is chosen as $\chi= 6 q^2$ and we take the time step to be $\Deltat= \calO(h^{(q+1)/2})$.
	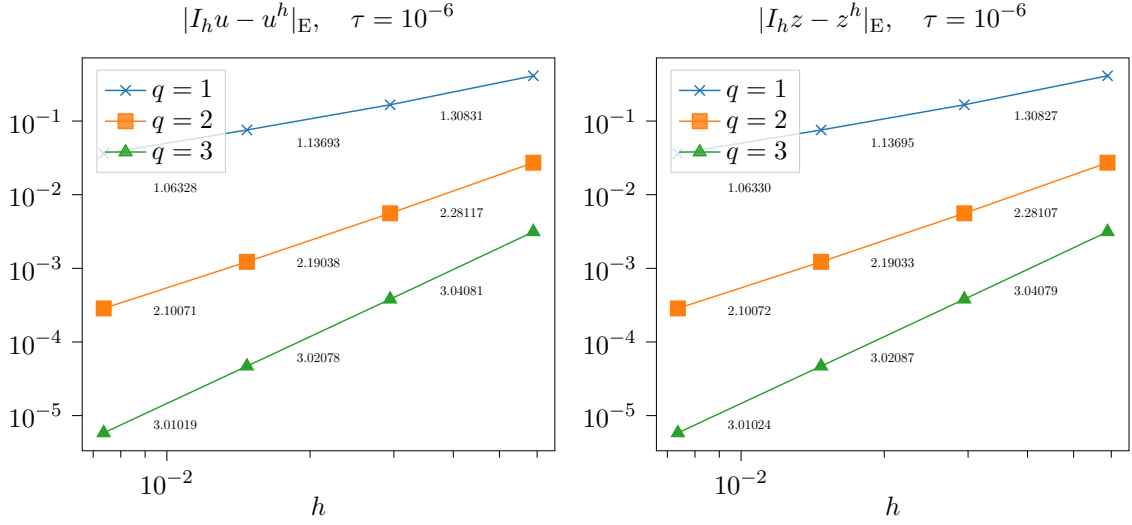
\begin{figure}[h]
		\centering
		\begin{subfigure}[b]{0.49\textwidth}
			\centering
			\resizebox{\textwidth}{!}{
\begin{tikzpicture}

\definecolor{darkgray176}{RGB}{176,176,176}
\definecolor{darkorange25512714}{RGB}{255,127,14}
\definecolor{forestgreen4416044}{RGB}{44,160,44}
\definecolor{lightgray204}{RGB}{204,204,204}
\definecolor{steelblue31119180}{RGB}{31,119,180}

\begin{axis}[
legend cell align={left},
legend style={
  fill opacity=0.8,
  draw opacity=1,
  text opacity=1,
  at={(0.03,0.97)},
  anchor=north west,
  draw=lightgray204
},
log basis x={10},
log basis y={10},
minor xtick={0.0002,0.0003,0.0004,0.0005,0.0006,0.0007,0.0008,0.0009,0.002,0.003,0.004,0.005,0.006,0.007,0.008,0.009,0.02,0.03,0.04,0.05,0.06,0.07,0.08,0.09,0.2,0.3,0.4,0.5,0.6,0.7,0.8,0.9,2,3,4,5,6,7,8,9},
tick align=outside,
tick pos=left,
title={\(\displaystyle \energyseminorm{I_h u-\uh}, \quad \tau=10^{-6}\)},
x grid style={darkgray176},
xlabel={\(\displaystyle h\)},
xmin=0.00663833660062115, xmax=0.0653820081580626,
xmode=log,
xtick style={color=black},
xtick={0.0001,0.001,0.01,0.1,1},
xticklabels={
  \(\displaystyle {10^{-4}}\),
  \(\displaystyle {10^{-3}}\),
  \(\displaystyle {10^{-2}}\),
  \(\displaystyle {10^{-1}}\),
  \(\displaystyle {10^{0}}\)
},
y grid style={darkgray176},
ymin=3.33092840191587e-06, ymax=0.720262487196451,
ymode=log,
ytick style={color=black},
ytick={1e-07,1e-06,1e-05,0.0001,0.001,0.01,0.1,1,10},
yticklabels={
  \(\displaystyle {10^{-7}}\),
  \(\displaystyle {10^{-6}}\),
  \(\displaystyle {10^{-5}}\),
  \(\displaystyle {10^{-4}}\),
  \(\displaystyle {10^{-3}}\),
  \(\displaystyle {10^{-2}}\),
  \(\displaystyle {10^{-1}}\),
  \(\displaystyle {10^{0}}\),
  \(\displaystyle {10^{1}}\)
}
]
\addplot [semithick, steelblue31119180, mark=x, mark size=3, mark options={solid}]
table {%
0.058925565098879 0.412091999259603
0.0294627825494395 0.166399666953334
0.0147313912747197 0.0756662185273275
0.00736569563735987 0.0362094790090931
};
\addlegendentry{$q=1$}
\addplot [semithick, darkorange25512714, mark=square*, mark size=3, mark options={solid}]
table {%
0.058925565098879 0.0271771432735682
0.0294627825494395 0.00559118427582754
0.0147313912747197 0.00122499149012683
0.00736569563735987 0.000285598129095356
};
\addlegendentry{$q=2$}
\addplot [semithick, forestgreen4416044, mark=triangle*, mark size=3, mark options={solid}]
table {%
0.058925565098879 0.00313283541410106
0.0294627825494395 0.000380683365399185
0.0147313912747197 4.69048945512287e-05
0.00736569563735987 5.82186205931615e-06
};
\addlegendentry{$q=3$}
\draw (axis cs:0.0416666666666667,0.183303753374437) node[
  scale=0.5,
  text=black,
  rotate=0.0,
  yshift=-4mm
]{1.30831};
\draw (axis cs:0.0208333333333333,0.0785462185318752) node[
  scale=0.5,
  text=black,
  rotate=0.0,
  yshift=-4mm
]{1.13693};
\draw (axis cs:0.0104166666666667,0.036640398909083) node[
  scale=0.5,
  text=black,
  rotate=0.0,
  yshift=-1cm
]{1.06328};
\draw (axis cs:0.0416666666666667,0.00862882865197892) node[
  scale=0.5,
  text=black,
  rotate=0.0,
  yshift=-4mm
]{2.28117};
\draw (axis cs:0.0208333333333333,0.00183196207581752) node[
  scale=0.5,
  text=black,
  rotate=0.0,
  yshift=-4mm
]{2.19038};
\draw (axis cs:0.0104166666666667,0.000414039957119598) node[
  scale=0.5,
  text=black,
  rotate=0.0,
  yshift=-4mm
]{2.10071};
\draw (axis cs:0.0416666666666667,0.000764449462720756) node[
  scale=0.5,
  text=black,
  rotate=0.0,
  yshift=-4mm
]{3.04081};
\draw (axis cs:0.0208333333333333,9.35382137129701e-05) node[
  scale=0.5,
  text=black,
  rotate=0.0,
  yshift=-4mm
]{3.02078};
\draw (axis cs:0.0104166666666667,1.15674618967244e-05) node[
  scale=0.5,
  text=black,
  rotate=0.0,
  yshift=-4mm
]{3.01019};
\end{axis}

\end{tikzpicture}}
		\end{subfigure}
		\begin{subfigure}[b]{0.49\textwidth}
			\centering
			\resizebox{\textwidth}{!}{
\begin{tikzpicture}

\definecolor{darkgray176}{RGB}{176,176,176}
\definecolor{darkorange25512714}{RGB}{255,127,14}
\definecolor{forestgreen4416044}{RGB}{44,160,44}
\definecolor{lightgray204}{RGB}{204,204,204}
\definecolor{steelblue31119180}{RGB}{31,119,180}

\begin{axis}[
legend cell align={left},
legend style={
  fill opacity=0.8,
  draw opacity=1,
  text opacity=1,
  at={(0.03,0.97)},
  anchor=north west,
  draw=lightgray204
},
log basis x={10},
log basis y={10},
minor xtick={0.0002,0.0003,0.0004,0.0005,0.0006,0.0007,0.0008,0.0009,0.002,0.003,0.004,0.005,0.006,0.007,0.008,0.009,0.02,0.03,0.04,0.05,0.06,0.07,0.08,0.09,0.2,0.3,0.4,0.5,0.6,0.7,0.8,0.9,2,3,4,5,6,7,8,9},
tick align=outside,
tick pos=left,
title={\(\displaystyle \energyseminorm{I_h z-\zh}, \quad \tau=10^{-6}\)},
x grid style={darkgray176},
xlabel={\(\displaystyle h\)},
xmin=0.00663833660062115, xmax=0.0653820081580626,
xmode=log,
xtick style={color=black},
xtick={0.0001,0.001,0.01,0.1,1},
xticklabels={
  \(\displaystyle {10^{-4}}\),
  \(\displaystyle {10^{-3}}\),
  \(\displaystyle {10^{-2}}\),
  \(\displaystyle {10^{-1}}\),
  \(\displaystyle {10^{0}}\)
},
y grid style={darkgray176},
ymin=3.33086658235264e-06, ymax=0.720247534649615,
ymode=log,
ytick style={color=black},
ytick={1e-07,1e-06,1e-05,0.0001,0.001,0.01,0.1,1,10},
yticklabels={
  \(\displaystyle {10^{-7}}\),
  \(\displaystyle {10^{-6}}\),
  \(\displaystyle {10^{-5}}\),
  \(\displaystyle {10^{-4}}\),
  \(\displaystyle {10^{-3}}\),
  \(\displaystyle {10^{-2}}\),
  \(\displaystyle {10^{-1}}\),
  \(\displaystyle {10^{0}}\),
  \(\displaystyle {10^{1}}\)
}
]
\addplot [semithick, steelblue31119180, mark=x, mark size=3, mark options={solid}]
table {%
0.058925565098879 0.412083485508657
0.0294627825494395 0.166401359132916
0.0147313912747197 0.0756659162393115
0.00736569563735987 0.0362089925793039
};
\addlegendentry{$q=1$}
\addplot [semithick, darkorange25512714, mark=square*, mark size=3, mark options={solid}]
table {%
0.058925565098879 0.0271796325333084
0.0294627825494395 0.00559206913469982
0.0147313912747197 0.0012252278653444
0.00736569563735987 0.000285651052010327
};
\addlegendentry{$q=2$}
\addplot [semithick, forestgreen4416044, mark=triangle*, mark size=3, mark options={solid}]
table {%
0.058925565098879 0.00313305899293057
0.0294627825494395 0.00038071424163314
0.0147313912747197 4.69057297809545e-05
0.00736569563735987 5.82175342752453e-06
};
\addlegendentry{$q=3$}
\draw (axis cs:0.0416666666666667,0.183302791882137) node[
  scale=0.5,
  text=black,
  rotate=0.0,
  yshift=-4mm
]{1.30827};
\draw (axis cs:0.0208333333333333,0.0785464610158015) node[
  scale=0.5,
  text=black,
  rotate=0.0,
  yshift=-4mm
]{1.13695};
\draw (axis cs:0.0104166666666667,0.0366400796097876) node[
  scale=0.5,
  text=black,
  rotate=0.0,
  yshift=-1cm
]{1.06330};
\draw (axis cs:0.0416666666666667,0.00862990661879828) node[
  scale=0.5,
  text=black,
  rotate=0.0,
  yshift=-4mm
]{2.28107};
\draw (axis cs:0.0208333333333333,0.00183228378672509) node[
  scale=0.5,
  text=black,
  rotate=0.0,
  yshift=-4mm
]{2.19033};
\draw (axis cs:0.0104166666666667,0.000414118265785414) node[
  scale=0.5,
  text=black,
  rotate=0.0,
  yshift=-4mm
]{2.10072};
\draw (axis cs:0.0416666666666667,0.000764507741921475) node[
  scale=0.5,
  text=black,
  rotate=0.0,
  yshift=-4mm
]{3.04079};
\draw (axis cs:0.0208333333333333,9.35428397980546e-05) node[
  scale=0.5,
  text=black,
  rotate=0.0,
  yshift=-4mm
]{3.02087};
\draw (axis cs:0.0104166666666667,1.15674569646996e-05) node[
  scale=0.5,
  text=black,
  rotate=0.0,
  yshift=-4mm
]{3.01024};
\end{axis}

\end{tikzpicture}}
		\end{subfigure}
		\caption{Plots of discrete errors for pressure and the combined quantity $z= \tau u_t+u$ with the relaxation parameter $\tau =10^{-6}$ }
		\label{fig: KnownSolution_tau=1e-6}
	\end{figure}
	\begin{figure}[h]
		\centering
		\begin{subfigure}[b]{0.49\textwidth}
			\centering
			\resizebox{\textwidth}{!}{
\begin{tikzpicture}

\definecolor{darkgray176}{RGB}{176,176,176}
\definecolor{darkorange25512714}{RGB}{255,127,14}
\definecolor{forestgreen4416044}{RGB}{44,160,44}
\definecolor{lightgray204}{RGB}{204,204,204}
\definecolor{steelblue31119180}{RGB}{31,119,180}

\begin{axis}[
legend cell align={left},
legend style={
  fill opacity=0.8,
  draw opacity=1,
  text opacity=1,
  at={(0.03,0.97)},
  anchor=north west,
  draw=lightgray204
},
log basis x={10},
log basis y={10},
minor xtick={0.0002,0.0003,0.0004,0.0005,0.0006,0.0007,0.0008,0.0009,0.002,0.003,0.004,0.005,0.006,0.007,0.008,0.009,0.02,0.03,0.04,0.05,0.06,0.07,0.08,0.09,0.2,0.3,0.4,0.5,0.6,0.7,0.8,0.9,2,3,4,5,6,7,8,9},
tick align=outside,
tick pos=left,
title={\(\displaystyle \energyseminorm{I_h u-\uh}, \quad \tau=1\)},
x grid style={darkgray176},
xlabel={\(\displaystyle h\)},
xmin=0.00663833660062115, xmax=0.0653820081580626,
xmode=log,
xtick style={color=black},
xtick={0.0001,0.001,0.01,0.1,1},
xticklabels={
  \(\displaystyle {10^{-4}}\),
  \(\displaystyle {10^{-3}}\),
  \(\displaystyle {10^{-2}}\),
  \(\displaystyle {10^{-1}}\),
  \(\displaystyle {10^{0}}\)
},
y grid style={darkgray176},
ymin=9.02794987352811e-07, ymax=0.162960761151248,
ymode=log,
ytick style={color=black},
ytick={1e-08,1e-07,1e-06,1e-05,0.0001,0.001,0.01,0.1,1,10},
yticklabels={
  \(\displaystyle {10^{-8}}\),
  \(\displaystyle {10^{-7}}\),
  \(\displaystyle {10^{-6}}\),
  \(\displaystyle {10^{-5}}\),
  \(\displaystyle {10^{-4}}\),
  \(\displaystyle {10^{-3}}\),
  \(\displaystyle {10^{-2}}\),
  \(\displaystyle {10^{-1}}\),
  \(\displaystyle {10^{0}}\),
  \(\displaystyle {10^{1}}\)
}
]
\addplot [semithick, steelblue31119180, mark=x, mark size=3, mark options={solid}]
table {%
0.058925565098879 0.0940051189574492
0.0294627825494395 0.0424488982760546
0.0147313912747197 0.019933265370761
0.00736569563735987 0.00965987525455922
};
\addlegendentry{$q=1$}
\addplot [semithick, darkorange25512714, mark=square*, mark size=3, mark options={solid}]
table {%
0.058925565098879 0.00628080192659537
0.0294627825494395 0.00137841327823235
0.0147313912747197 0.000317919106488498
0.00736569563735987 7.59472524887307e-05
};
\addlegendentry{$q=2$}
\addplot [semithick, forestgreen4416044, mark=triangle*, mark size=3, mark options={solid}]
table {%
0.058925565098879 0.000822194338193781
0.0294627825494395 0.000101652870200046
0.0147313912747197 1.25900945818854e-05
0.00736569563735987 1.56502284060869e-06
};
\addlegendentry{$q=3$}
\draw (axis cs:0.0416666666666667,0.0442188051478786) node[
  scale=0.5,
  text=black,
  rotate=0.0,
  yshift=-4mm
]{1.14701};
\draw (axis cs:0.0208333333333333,0.0203620020006921) node[
  scale=0.5,
  text=black,
  rotate=0.0,
  yshift=-4mm
]{1.09055};
\draw (axis cs:0.0104166666666667,0.00971343913759757) node[
  scale=0.5,
  text=black,
  rotate=0.0,
  yshift=-1cm
]{1.04510};
\draw (axis cs:0.0416666666666667,0.0020596589472647) node[
  scale=0.5,
  text=black,
  rotate=0.0,
  yshift=-4mm
]{2.18794};
\draw (axis cs:0.0208333333333333,0.000463389382394417) node[
  scale=0.5,
  text=black,
  rotate=0.0,
  yshift=-4mm
]{2.11628};
\draw (axis cs:0.0104166666666667,0.000108770816394941) node[
  scale=0.5,
  text=black,
  rotate=0.0,
  yshift=-4mm
]{2.06559};
\draw (axis cs:0.0416666666666667,0.000202369520991716) node[
  scale=0.5,
  text=black,
  rotate=0.0,
  yshift=-4mm
]{3.01583};
\draw (axis cs:0.0208333333333333,2.50421930482529e-05) node[
  scale=0.5,
  text=black,
  rotate=0.0,
  yshift=-4mm
]{3.01329};
\draw (axis cs:0.0104166666666667,3.10722624492915e-06) node[
  scale=0.5,
  text=black,
  rotate=0.0,
  yshift=-4mm
]{3.00803};
\end{axis}

\end{tikzpicture}}
		\end{subfigure}
		\begin{subfigure}[b]{0.49\textwidth}
			\centering
			\resizebox{\textwidth}{!}{
\begin{tikzpicture}

\definecolor{darkgray176}{RGB}{176,176,176}
\definecolor{darkorange25512714}{RGB}{255,127,14}
\definecolor{forestgreen4416044}{RGB}{44,160,44}
\definecolor{lightgray204}{RGB}{204,204,204}
\definecolor{steelblue31119180}{RGB}{31,119,180}

\begin{axis}[
legend cell align={left},
legend style={
  fill opacity=0.8,
  draw opacity=1,
  text opacity=1,
  at={(0.97,0.03)},
  anchor=south east,
  draw=lightgray204
},
log basis x={10},
log basis y={10},
minor xtick={0.0002,0.0003,0.0004,0.0005,0.0006,0.0007,0.0008,0.0009,0.002,0.003,0.004,0.005,0.006,0.007,0.008,0.009,0.02,0.03,0.04,0.05,0.06,0.07,0.08,0.09,0.2,0.3,0.4,0.5,0.6,0.7,0.8,0.9,2,3,4,5,6,7,8,9},
tick align=outside,
tick pos=left,
title={\(\displaystyle \energyseminorm{I_h z-\zh}, \quad \tau=1\)},
x grid style={darkgray176},
xlabel={\(\displaystyle h\)},
xmin=0.00663833660062115, xmax=0.0653820081580626,
xmode=log,
xtick style={color=black},
xtick={0.0001,0.001,0.01,0.1,1},
xticklabels={
  \(\displaystyle {10^{-4}}\),
  \(\displaystyle {10^{-3}}\),
  \(\displaystyle {10^{-2}}\),
  \(\displaystyle {10^{-1}}\),
  \(\displaystyle {10^{0}}\)
},
y grid style={darkgray176},
ymin=7.45722198752268e-06, ymax=1.02671195204153,
ymode=log,
ytick style={color=black},
ytick={1e-07,1e-06,1e-05,0.0001,0.001,0.01,0.1,1,10,100},
yticklabels={
  \(\displaystyle {10^{-7}}\),
  \(\displaystyle {10^{-6}}\),
  \(\displaystyle {10^{-5}}\),
  \(\displaystyle {10^{-4}}\),
  \(\displaystyle {10^{-3}}\),
  \(\displaystyle {10^{-2}}\),
  \(\displaystyle {10^{-1}}\),
  \(\displaystyle {10^{0}}\),
  \(\displaystyle {10^{1}}\),
  \(\displaystyle {10^{2}}\)
}
]
\addplot [semithick, steelblue31119180, mark=x, mark size=3, mark options={solid}]
table {%
0.058925565098879 0.59960265530021
0.0294627825494395 0.316590456559783
0.0147313912747197 0.163163436071043
0.00736569563735987 0.0828884025189548
};
\addlegendentry{$q=1$}
\addplot [semithick, darkorange25512714, mark=square*, mark size=3, mark options={solid}]
table {%
0.058925565098879 0.0497529251938092
0.0294627825494395 0.0117947453565127
0.0147313912747197 0.00315587945833633
0.00736569563735987 0.00079701120021143
};
\addlegendentry{$q=2$}
\addplot [semithick, forestgreen4416044, mark=triangle*, mark size=3, mark options={solid}]
table {%
0.058925565098879 0.00596203691489178
0.0294627825494395 0.000782754592384649
0.0147313912747197 0.000100177267925034
0.00736569563735987 1.27691544991291e-05
};
\addlegendentry{$q=3$}
\draw (axis cs:0.0416666666666667,0.304985170810019) node[
  scale=0.5,
  text=black,
  rotate=0.0,
  yshift=-4mm
]{0.92139};
\draw (axis cs:0.0208333333333333,0.159095674022273) node[
  scale=0.5,
  text=black,
  rotate=0.0,
  yshift=-4mm
]{0.95630};
\draw (axis cs:0.0104166666666667,0.0814059869853678) node[
  scale=0.5,
  text=black,
  rotate=0.0,
  yshift=-4mm
]{0.97708};
\draw (axis cs:0.0416666666666667,0.0169571020775154) node[
  scale=0.5,
  text=black,
  rotate=0.0,
  yshift=-4mm
]{2.07664};
\draw (axis cs:0.0208333333333333,0.0042707340525481) node[
  scale=0.5,
  text=black,
  rotate=0.0,
  yshift=-4mm
]{1.90203};
\draw (axis cs:0.0104166666666667,0.0011101724751846) node[
  scale=0.5,
  text=black,
  rotate=0.0,
  yshift=-4mm
]{1.98537};
\draw (axis cs:0.0416666666666667,0.00151219633969871) node[
  scale=0.5,
  text=black,
  rotate=0.0,
  yshift=-4mm
]{2.92917};
\draw (axis cs:0.0208333333333333,0.000196017769845556) node[
  scale=0.5,
  text=black,
  rotate=0.0,
  yshift=-4mm
]{2.96600};
\draw (axis cs:0.0104166666666667,2.50359284949321e-05) node[
  scale=0.5,
  text=black,
  rotate=0.0,
  yshift=-4mm
]{2.97182};
\end{axis}

\end{tikzpicture}}
		\end{subfigure}
		\caption{Plots of discrete errors for pressure and the combined quantity $z= \tau u_t+u$ with the relaxation parameter $\tau=1$}
		\label{fig: KnownSolution_tau=1}
	\end{figure}
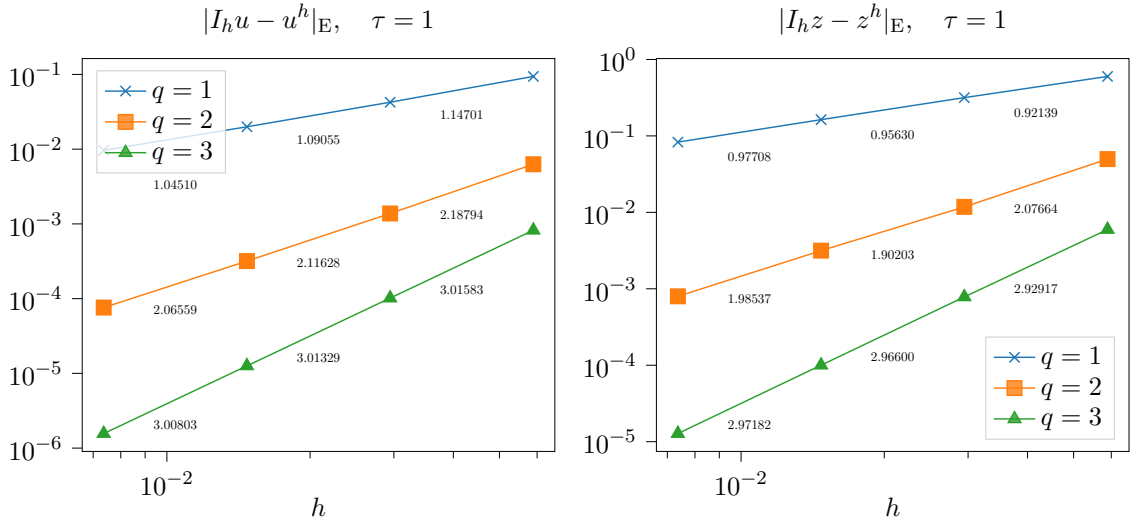
	
	Figures~\ref{fig: KnownSolution_tau=1e-6} and~\ref{fig: KnownSolution_tau=1} display empirically obtained convergence rates of the discrete errors $\uh- \interpolant u$ and $\zh-\interpolant z$, measured in the energy semi-norm $\energyseminorm{\cdot}$, for the pressure and combined quantity $z$ for polynomial degrees $q \in \{1,2,3\}$. The observed rates match the expected orders based on Theorem~\ref{thm: main} and \eqref{conv rate z}. 
	\subsubsection{Focused ultrasound}
	In the second experiment, we consider a setting motivated by therapeutic applications of focused ultrasound. The domain $\Omega$ is the rectangle $(0, l_x) \times ( -l_y/2, l_y/2)$, where $l_x = \SI{6.4}{cm}$ and $l_y =\SI{4.8}{cm}$, with an arc protrusion on the left side with depth $\SI{0.26}{cm}$, representing the curved transducer surface; see~Figure~\ref{fig: SnapshotFUS}.   On the arc part of the boundary, we prescribe Neumann data
	\begin{eq}
		\csq	\frac{\partial u}{\partial n}+ b 	\frac{\partial u_t}{\partial n} = \gN \quad \text{on }\ \GN  \quad \text{with} \quad 	\gN(t) = a \sin(2 \pi \omega t),
	\end{eq}
\noindent modeling a sinusoidally driven transducer. We choose the source frequency and amplitude as
	\begin{eq}
		\omega= \SI{0.5}{MHz}, \qquad a= \SI{1.5e9}{Pa/m}.
	\end{eq} 
	On the rest of the boundary $\Gabs= \partial \Omega \setminus \GN$, we prescribe absorbing boundary conditions $c(\tau \ut+u)_t + 	\csq	\frac{\partial u}{\partial n}+ b 	\frac{\partial u_t}{\partial n} =0$, that is,  \eqref{tau abcs} with $\alpha=c$ and $\gabs=0$. The acoustic medium parameters are chosen as
	\begin{eq} \label{med par}
		c = \SI{1500}{m/s}, \qquad b = \SI{4e-6}{m^2/s}. 
	\end{eq}
	The nonlinearity coefficient is taken to be spatially varying:
	\begin{eq}
		k(x, y) = \frac{\betaa(x, y)}{\rho \csq} \quad \text{with} \quad
		\betaa(x,y) = \begin{cases}
			0,    & x \leq \SI{0.5}{cm}, \\
			6.25, & \SI{0.5}{cm} \leq x \leq \SI{2}{cm}, \\
			4.5,  & \SI{2}{cm}  \leq x \leq \SI{6.4}{cm}.
		\end{cases}
	\end{eq}
	The simulation is run until $T=\SI{62}{\mu s}$  using polynomial degree $q=1$, with $\Deltat = \calO(h^{(q+1)/2})$, Newmark parameters and dG stabilization parameter $\chi$ as in Section~\ref{Sec: numerics known solution}.  Initial conditions are set to zero.
		
	\begin{figure}[h]
		\centering
		\includegraphics[scale=0.28, trim=0 4cm 1cm 1.2cm, clip]{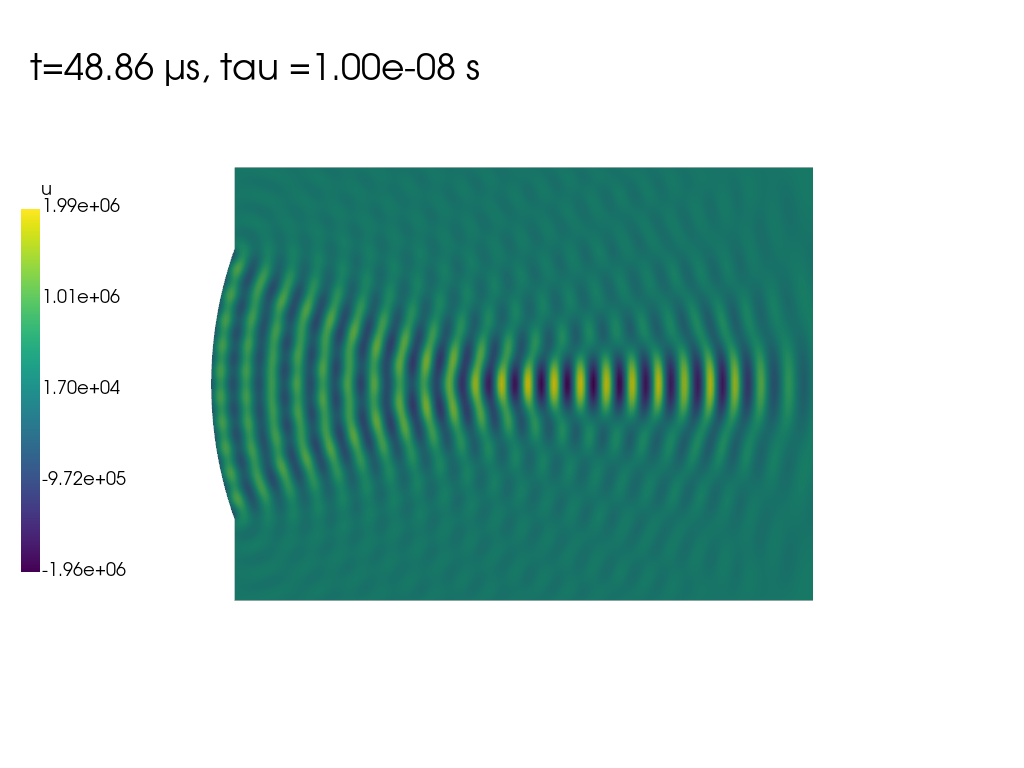}
		\hspace*{-1.8cm}
		\caption{ Snapshot of the pressure field with the relaxation parameter $\tau =\SI{10}{ns}$}
		\label{fig: SnapshotFUS}
	\end{figure}
	Figure~\ref{fig: SnapshotFUS} shows a snapshot of the focused pressure field at around $80\%$ of the final time with the relaxation time $\tau =\SI{10}{n s}$, where we can observe focusing of sound waves. 
	
	Figure~\ref{fig: Comparison_different_tau}	shows on the left plot a comparison of the pressure along the axis of symmetry at $t= 80\% \, T$ for two different values of the relaxation parameter: $\tau = \SI{10}{ns}$ and $\tau = \SI{1}{\mu s}$. The two settings correspond to the regimes $\tau \omega \ll 1$ and $\tau \omega = \calO(1)$. We observe that $\tau =\SI{1}{\mu s}$ introduces a significant reduction in amplitude, and that there is a phase shift between the two settings. These observations are consistent with the linear frequency-domain analysis; see, for example,~\cite[Ch.~4.1]{rudenko1977} and~\cite[Fig.~4.1]{rudenko1977}. Furthermore, as the nonlinearity generates higher harmonics with larger 
	effective frequencies, these are more strongly attenuated.
	
		\begin{figure}[h]
		{\scalebox{0.8}{\input{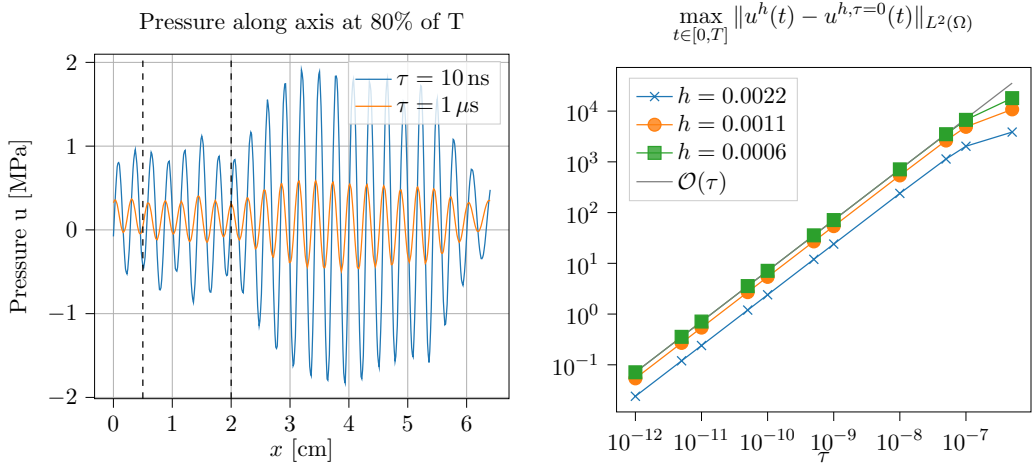}}} \hspace*{1mm}	{\scalebox{0.8}{
\begin{tikzpicture}

\definecolor{darkgray176}{RGB}{176,176,176}
\definecolor{darkorange25512714}{RGB}{255,127,14}
\definecolor{forestgreen4416044}{RGB}{44,160,44}
\definecolor{gray}{RGB}{128,128,128}
\definecolor{lightgray204}{RGB}{204,204,204}
\definecolor{steelblue31119180}{RGB}{31,119,180}

\begin{axis}[
legend cell align={left},
legend style={
  fill opacity=0.8,
  draw opacity=1,
  text opacity=1,
  at={(0.03,0.97)},
  anchor=north west,
  draw=lightgray204
},
log basis x={10},
log basis y={10},
tick align=outside,
tick pos=left,
title={\(\displaystyle  \max_{t \in [0,T]} \|u^h(t)-u^{h, \tau=0}(t)\|_{L^2(\Omega)}\)},
x grid style={darkgray176},
xlabel={\(\displaystyle \tau\)},
xmin=5.18861563251408e-13, xmax=9.63648177881565e-07,
xmode=log,
xtick style={color=black},
xtick={1e-14,1e-13,1e-12,1e-11,1e-10,1e-09,1e-08,1e-07,1e-06,1e-05},
xticklabels={
  \(\displaystyle {10^{-14}}\),
  \(\displaystyle {10^{-13}}\),
  \(\displaystyle {10^{-12}}\),
  \(\displaystyle {10^{-11}}\),
  \(\displaystyle {10^{-10}}\),
  \(\displaystyle {10^{-9}}\),
  \(\displaystyle {10^{-8}}\),
  \(\displaystyle {10^{-7}}\),
  \(\displaystyle {10^{-6}}\),
  \(\displaystyle {10^{-5}}\)
},
y grid style={darkgray176},
ymin=0.0117795989434621, ymax=72476.7307520947,
ymode=log,
ytick style={color=black},
ytick={0.001,0.01,0.1,1,10,100,1000,10000,100000,1000000},
yticklabels={
  \(\displaystyle {10^{-3}}\),
  \(\displaystyle {10^{-2}}\),
  \(\displaystyle {10^{-1}}\),
  \(\displaystyle {10^{0}}\),
  \(\displaystyle {10^{1}}\),
  \(\displaystyle {10^{2}}\),
  \(\displaystyle {10^{3}}\),
  \(\displaystyle {10^{4}}\),
  \(\displaystyle {10^{5}}\),
  \(\displaystyle {10^{6}}\)
}
]
\addplot [semithick, steelblue31119180, mark=x, mark size=3, mark options={solid}]
table {%
1e-12 0.0239731171976139
5e-12 0.119865487365832
1e-11 0.239730728064071
5e-11 1.19864374251348
1e-10 2.39726258237479
5e-10 11.9852850118161
1e-09 23.967839686299
1e-08 238.860537571715
5e-08 1138.59999514177
1e-07 2031.44385342149
5e-07 3860.53203550446
};
\addlegendentry{$h=0.0022$}
\addplot [semithick, darkorange25512714, mark=*, mark size=3, mark options={solid}]
table {%
1e-12 0.0540340203227818
5e-12 0.270170028158215
1e-11 0.540339872526749
5e-11 2.70169199291085
1e-10 5.40336546717335
5e-10 27.0160637543924
1e-09 54.0300733671529
1e-08 539.491256741763
5e-08 2627.29224260214
1e-07 4893.30746554668
5e-07 10806.3153357877
};
\addlegendentry{$h=0.0011$}
\addplot [semithick, forestgreen4416044, mark=square*, mark size=3, mark options={solid}]
table {%
1e-12 0.071225349123804
5e-12 0.356126704271462
1e-11 0.712253305069089
5e-11 3.56126235227934
1e-10 7.12251409896679
5e-10 35.6121111265278
1e-09 71.2228933936975
1e-08 711.636579571119
5e-08 3505.65033236206
1e-07 6725.22228549195
5e-07 18008.7194015281
};
\addlegendentry{$h=0.0006$}
\addplot [semithick, gray, dash pattern=]
table {%
1e-12 0.071225349123804
5e-12 0.35612674561902
1e-11 0.71225349123804
5e-11 3.5612674561902
1e-10 7.1225349123804
5e-10 35.612674561902
1e-09 71.2253491238041
1e-08 712.25349123804
5e-08 3561.2674561902
1e-07 7122.5349123804
5e-07 35612.674561902
};
\addlegendentry{$\calO(\tau)$}
\end{axis}

\end{tikzpicture}}}
		\caption{({\bf left}) Pressure along the axis at $t= 80\%\, T$ for two different values of the relaxation parameter\  ({\bf right}) The vanishing $\tau$ convergence for different mesh sizes}
		\label{fig: Comparison_different_tau}
	\end{figure}

	The right plot in Figure~\ref{fig: Comparison_different_tau} illustrates the $\tau$ convergence by tracking the error
	\[
	\max_{t \in [0,T]} \| \uh(t) - \uhtauzero(t)\|_{\Ltwo}.
	\]
	The numerically observed rate is in agreement with the theoretically predicted linear order $\mathcal{O}(\tau)$ established in Theorem~\ref{thm: tau limit}, and the clustering of curves for finer spatial resolutions is consistent with $h$ convergence.
\section{Proof of Theorem~\ref{thm: main}} \label{Sec: proof}
This section is dedicated to proving Theorem~\ref{thm: main}, which is the backbone of the $\tau$ limiting analysis of the semi-discrete problem via the uniform bounds it provides. The proof will exploit some background results from the dG methodology, which we first recall.
\subsection{Background results on dG methods} \label{sec: dG discretization}
The following result contains  inverse and discrete trace inequalities in $\Vh$. Below we employ the notion of the mesh regularity parameter as introduced in~\cite[Def.~1.38]{di2011mathematical}.
\begin{lemma}[see~ {\cite[Lemmas~1.46 and 1.50 and Remark 1.47]{di2011mathematical}}]\label{lem: inverse ineq}
	Let the assumptions on $\triangles$ in Section~\ref{Sec: main results} hold. For $1 \leq \ell,\ell' \leq \infty$, $\phih \in \Vh$, and $K \in \triangles$:
	\begin{eq}\label{inverse inequality}
		\| \phih \|_{L^{\ell}(K)} \leq \invconstant \hK^{d(1/\ell-1/\ell')}\| \phih \|_{L^{\ell'}(K)},
	\end{eq}
	where $\invconstant$ depends on the mesh regularity parameter, $d$, $q$, $\ell$, $\ell'$. For $\phih \in \Vhq$ and face $F_K \in \calFh$ of $K \in \triangles$:
	\begin{eq}\label{face element trace ineq}
		\Ltwonorm{\phih}{F_K} \leq \Ctr h_K^{-1/2} \Ltwonorm{\phih}{K}, \quad
		\Ltwonorm{\phih}{\partial K} \leq \Ctr h_K^{-1/2} \Npartial^{-1/2} \Ltwonorm{\phih}{K},
	\end{eq}
	where $\Ctr$ depends on the mesh regularity, $d$, $q$, and $\Npartial = \displaystyle \max_{K \in \triangles} \mathrm{card}(F_K) = d+1$ is the maximum number of mesh faces composing the boundary of mesh elements.
\end{lemma}
In what follows, we will also need the following continuous trace inequality.
\begin{lemma}[Continuous trace inequality, see~{\cite[Lemma 1.49]{di2011mathematical}}] \label{lem: cont trace ineq}
	Let the assumptions made on $\triangles$ in Section~\ref{Sec: main results} hold.	Then, for all $v \in H^1(\triangles)$, all $K \in \triangles$, and all $F \in \calFh$,
	\begin{equation}  \label{cont trace ineq}
		\|v\|^2_{L^2(F)} \leq C_{\mathrm{cti}}\left(2\|\nabla v\|_{\LtwoK} + d\, \hK^{-1}\|v\|_{L^2(K)}\right)\|v\|_{L^2(K)},
	\end{equation}
	where the constant $C_{\mathrm{cti}}$ depends on the mesh regularity parameter.
\end{lemma}
In the uniform error analysis of \eqref{semi-discrete jmgt}, we frequently rely on 
coercivity and boundedness of $\asip$ defined in \eqref{def: ah}. If $\chi > \Ctr^2(d+1)$, 
then
\begin{eq}\label{coercivity}
	\asip(\phih,\phih) \gtrsim \dgseminorm{\phih}^{2}, 
	\quad \forall\, \phih \in \Vh,
\end{eq}
which follows similarly to~\cite[Lemma 4.12]{di2011mathematical}; see also~\cite[Lemma 2.6]{dewit2026}.
For functions $\psi \in H^{2}(\Omega)+\Vh$, boundedness holds in the following sense:
\begin{eq}\label{boundedness dGstar}
	|\asip(\psi, \phih)| \lesssim \dgstarseminorm{\psi}\dgseminorm{\phih},
	\quad \forall\, \phih \in \Vh,
\end{eq}
where we have introduced
	\begin{eq} \label{def dgstarseminorm}
			\dgstarseminorm{\psi} = \left( \dgseminorm{\psi}^{2} + \sum_{K\in \triangles} \hK \Ltwonorm{\nabla \psi \cdot n}{\partial K}^{2}\right)^{1/2};
		\end{eq}
see~\cite[Lemma 4.16]{di2011mathematical} and~\cite[Lemma 2.7]{dewit2026}.		
For functions $\phih \in \Vhq$, the equivalence of $\dgseminorm{\phih}$ and $\dgstarseminorm{\phih}$ can be argued as in \cite[Lemma 4.20]{di2011mathematical} so that
\begin{eq}\label{boundedness}
	|\asip(\psih, \phih)| \lesssim \dgseminorm{\psih}\dgseminorm{\phih},
	\quad \forall\, \psih, \, \phih \in \Vh.
\end{eq}
We also recall properties of the local interpolant $\locainterpolant$,  from which the corresponding global estimates for $\interpolant$  follow via~\eqref{def interpolant}.
\begin{lemma}[see~{\cite[Lemma 4.4.1, Theorem 4.4.4, Corollary 4.4.7]{brennerscott}}] 
	\label{lem: local interpolant properties}
	Let the assumptions on $\triangles$ in Section~\ref{Sec: main results} hold. 
	For $n-d/\ell>0$, $1 < \ell \leq \infty$, and $0 \leq i \leq n \leq q+1$, the following approximation and stability bounds hold 
	with constants independent of $K$ and $h$:
	\begin{eq} \label{est interpolant approx Linf}
		|\phi-\locainterpolant \phi|_{W^{i,\ell}(K)} &\lesssim h^{n-i}
		|\phi|_{W^{n,\ell}(K)},  \quad 
		| \phi - \locainterpolant \phi |_{L^{\infty}(K)} \lesssim h^{n-d/2}
		|\phi|_{H^{n}(K)}, \\
		\| \locainterpolant \phi \|_{L^\infty(K)} &\lesssim  
		\| \phi \|_{L^{\infty}(K)}.
	\end{eq}
\end{lemma}
We will also exploit the following approximation result in the error analysis:
\begin{eq} \label{est interpolant dGstar}
	\dgstarseminorm{\phi - \Ih \phi } \lesssim h^{n-1}\|\phi\|_{H^{n}({\Om})}, \quad \phi \in H^n(\Omega), \quad  2 \leq n \leq q+1,
\end{eq}
which follows by combining Lemma~\ref{lem:  local interpolant properties} with the continuous trace inequality. \\

\noindent \textbf{Idea of the proof of Theorem~\ref{thm: main}}. The proof of Theorem~\ref{thm: main} contains several steps, following the general well-posedness and error analysis framework put forward in~\cite{hochbruck22, maier2020error} for quasilinear wave-type problems; we refer also to~\cite{dorich2025strong, dorich2024robust, dewit2026} for its subsequent use in different contexts.  \\
\indent To establish existence, we rewrite \eqref{semi-discrete jmgt} as a system of nonlinear first-order differential equations with a locally Lipschitz continuous right-hand side and such that 
\[
(\uh(0),\uht(0), \uhtt(0)) \in \openset,
\]
where $\openset$ is a suitably defined open set. We will then show that there exists large enough $C_{0}>0$, independent of $h$ and $\tau$, such that
\begin{align}\label{eq: def thh}
	\begin{split}
		\thh \coloneqq  \sup \,  \Bigl \{  t \in (0,T] \mid & \textup{  a unique solution } 
		\uh \in C^{3}([0,t];\Vhqp)\ \textup{of \eqref{semi-discrete jmgt} exists and }  \\ &  
		\tnorm{\uh(t)-\Ih u(t)} \leq \Czero h^{d/2} \Bigr \}
	\end{split}
\end{align}
is positive, where $\tnorm{\cdot}$ is defined in \eqref{def calE}.  The central and most challenging component of the proof is to then establish error bounds on the interval $[0,\thh]$ that do not depend on $\thh$. This will allow us to show that 
\[
(\uh(\thh),\uht(\thh), \uhtt(\thh)) \in \openset
\]
and conclude that the solution exists beyond $\thh$ and thus $\thh = T$. 
\subsection*{Step I: Existence with a discretization-dependent horizon}  The general outline of the arguments for showing existence until a possibly $h$-dependent final time are analogous to those of~\cite[Section 3.1]{dewit2026}, where the dG semi-discretization of the limiting Westervelt equation is studied with absorbing boundary conditions. We begin by proving that the approximate initial conditions belong to an open set on which the semi-discrete problem does not degenerate.
\begin{lemma}\label{lem: p0 in D}
	Under the assumptions of Theorem \ref{thm: main} and with $r_{0} \in (0, 1-r)$ a fixed constant, the approximate initial data $(\uzeroh, \uoneh, \utwoh)$ lie in  $\openset$, where
	\begin{align} \label{def openset}
		\openset = \{ (\psih, \phih, \qh)^T \in (\Vhq)^{3} \mid \| k \psi_h \|_{L^{\infty}(\Omega)} < r+r_{0}<1 \},
	\end{align}
\end{lemma}
\begin{proof}
	Using \eqref{est interpolant approx Linf}, we can estimate  
	\begin{eq}
		\| k \uzeroh\|_{L^{\infty}(\Omega)} = \|k\Ih \uzero \|_{L^{\infty}(\Omega)} 
		&\leq  \| k \uzero \|_{L^{\infty}(\Omega)} +  \Linfnormk \| \Ih \uzero - \uzero \|_{L^{\infty}(\Omega)} \\ 
		&\leq r + C \Linfnormk h^{q+1-d/2}\|\uzero\|_{H^{q+1}(\Omega)},
	\end{eq}
	for some $C>0$, independent of $h$. Provided $\Linfnormk \leq \bark$ is small enough, we then have $ \| k \uzeroh \|_{L^{\infty}(\Omega)} < r+r_{0} < 1$. Since the set $\openset$ involves a condition only on the first component, we conclude that $(\uzeroh, \uoneh, \utwoh) \in \openset$.
\end{proof}
We next prove that a semi-discrete solution exists up to final time $\thh$.
\begin{proposition}
	Under the assumptions of Theorem \ref{thm: main}, $\thh>0$.
\end{proposition}\vspace*{-4mm}
\begin{proof}    
	To show existence of semi-discrete solutions we write the semi-discrete acoustic problem as a first-order system of ODEs. Following, e.g.,~\cite[Sec.~3.1]{dewit2026}, we introduce the discrete multiplication operator $\Lambdah$ as 
	\begin{eq}
		\left(\Lambdah(\uh) \psih, \phih\right)_{\Ltwo} =  \intO (1+k \uh) \psih \phih \dx
	\end{eq}
	for all $(\psih, \phih) \in  \Vh \times \Vh$.  By Lemma \ref{lem: p0 in D}, we have $ \| k \uh(0)\|_{L^{\infty}(\Omega)} < 1$, which implies that the operator $\Lambda_{h}$ is locally invertible at $t=0$. \\
	\indent We also define the discrete differential operator $\Ahsip: \Vhqp \rightarrow \Vhqp$, such that 
	\begin{eq}
		\left(\Ahsip \psih, \phih\right)_{\Ltwo} = \asip(\psih, \phih),
	\end{eq}
	for all $(\psih, \phih) \in  \Vh \times \Vh$. The properties of $\Ahsip$ 
	follow from the consistency, discrete coercivity, and boundedness of  $\asip$; cf.~\cite[Lemma 4.71]{di2011mathematical}. Furthermore, we introduce the extension operators$\Gh_{\textup{abs}}: \VGhabs \rightarrow \Vhq$ and $\Gh_{\textup{N}}: \VGhN \rightarrow \Vhq$ for the boundary integral terms: 
	\begin{eq}
		(\Gh_{\textup{abs}} \psih, \phih)_{L^2(\Omega)} = \int_{\Gabs} \psih \phih \dG, \quad 	(\Gh_{\textup{N}} \psih, \phih)_{L^2(\Omega)} = \int_{\GN} \psih \phih \dG
	\end{eq}
	for all $\psih \in \VGh, \; \phih \in \Vh$. 
	We can then rewrite the semi-discrete acoustic problem as the first-order system of ODEs:
	\begin{align}\label{eq: nonlin system of odes}
		\begin{bmatrix} \uh_{t} \\ \ph_t \\ \qh_{t}  \end{bmatrix} = \begin{bmatrix}  \ph\\ \qh \\ F(\uh, \ph, \qh) \end{bmatrix}, \qquad 	\begin{bmatrix} \uh \\ \ph \\ \qh \end{bmatrix}(0) = \begin{bmatrix} \Ih \uzero \\ \Ih \uone \\ \Ih \utwo \end{bmatrix},
	\end{align}
	where the function $F$ is defined as
	\begin{eq}
		F(\uh, \ph, \qh)=&\, \begin{multlined}[t] \Lambda_{h}^{-1}(\uh)\left( \fh -  \pi_{h}(k (\ph)^{2}) - \csq \Ahsip \uh - \btau \Ahsip \ph \right. \\ \left. \hspace*{2cm}-\alpha \Gh_{\textup{abs}}( {\ph_{\vert \Ga}})-\alpha \tau \Gh_{\textup{abs}}( {\qh_{\vert \Ga}} ) + \Gh_{\textup{abs}} \ghabs +\Gh_{\textup{N}} \ghN \right),
		\end{multlined}
	\end{eq}
	and $\pi_{h}$ is the $L^{2}(\Omega)$ projection onto $\Vhqp$.   \\
	\indent To apply the local version of the Picard--Lindel\"of theorem  to \eqref{eq: nonlin system of odes}, we introduce a ball centered at the discrete initial data and contained in $\openset$: let $\bar{\rho}>0$ be small enough so that 
	\[
	\ball ((\uzeroh, \uoneh, \utwoh)^{T}; \bar{\rho})=\left \{(\psih, \phih,  \qh)^T \in (\Vhq)^3:  \|(\psih-\uzeroh, \phih-\uoneh, \qh-\utwoh)\|_{\Linf} < \bar{\rho} \right\}
	\]
	satisfies
	\[
	\closedball ((\uzeroh, \uoneh, \utwoh)^{T}; \bar{\rho}) \subset \openset.
	\]
	Since $\|k \psih \|_{L^{\infty}(\Omega)} < r + r_0 < 1$ for all $(\psih, \phih, \qh)^T \in \ball$,  the operator $\Lambdah$ is invertible on $\ball$.
	\\
	\indent	We have established that the approximate initial conditions belong to $\openset$ in Lemma~\ref{lem: p0 in D}. The local Lipschitz continuity and boundedness of $F$ in $\closedball$ can be shown by using the fact that $\Vhq$ is finite-dimensional, and thus all norms are equivalent, analogously to the arguments in, for example,~\cite[Lemma 3.3]{hochbruck22}. Therefore, the local version of the Picard--Lindelöf theorem (see~\cite[Problem 7.1.3]{picardlindelofp105}) yields the local existence of a unique solution $\uh \in C^3([0,\tilde{t}];\Vh)$ for some $\tilde{t}>0$ with $(\uh{(t)},\uht{(t)}, \uhtt(t)) \in 
	\closedball \subseteq \openset$ for $t \in [0,\tilde{t}]$. \\
	\indent Furthermore, since $\tnorm{\uzeroh- \Ih \uzero} =0 \leq C_0 h^{d/2}$, the continuity in time of the local solution and its time derivatives implies that up to a possibly $h$-dependent final time $0<\thh\leq \tilde{t}$, the solution satisfies 
	\begin{eq}
		\tnorm{\uh(t)- \Ih u(t)} \leq C_0 h^{d/2}, \qquad t \in [0, \thh].
	\end{eq}
	Therefore, we have $\thh > 0$, which concludes the proof.
\end{proof}
\subsection*{Step II: Establishing uniform bounds}
The next task is crucial: we need to establish the $h$-uniform error bounds on $[0,\thh]$. In addition, the bounds should be uniform in $\tau$ to be able to also analyze the $\tau \rightarrow 0$ limiting behavior of semi-discrete solutions. 

To set up the arguments, we first split the total approximation error into a discrete 
part $\ehu = \uh - \Ih u$ and an interpolation part $\eIu = \Ih u - u$, so that
\[
\uh-u = (\uh-\Ih u)+ (\Ih u -u) = \ehu + \eIu.
\]
By subtracting the problems solved by $\uh$ and $u$, and splitting the error, we see that the discrete pressure error $\ehu$ satisfies
\begin{eq}\label{error eq}
	&\begin{multlined}[t] \intO \tau \ehu_{ttt} \phih \dx+	\intO (1+k \uh)\ehu_{tt} \phih \dx		+ \csq\asip(\ehu+\tau \ehu_t, \phih) \\
		 + \intO k ((\uht+\ut) \ehu_{t}+ \utt \ehu) \phih \dx 
		 + \int_{\Gabs} \alpha (\ehu_t+\tau \ehu_{tt}) \phih \dG 
		=\,  - \intO \deltahp \phih \dx 		
	\end{multlined}
\end{eq}
for all $\phih \in \Vh$, $t \in [0,T]$. Here we have introduced the defect $\deltah$, which satisfies
	\begin{eq}\label{defect 1}
	\begin{multlined}[t] \intO \deltah \phih \dx  =\tau \intO \eI_{ttt} \phih \dx+ \intO  (1+k \uh) \eIu_{tt} \phih\dx+ \intO k   \utt \eIu  \phih \dx \\
		+ \intO \left( k \uh_{t}\eIu_{t} + k  \ut \eIu_{t} \right) \phih \dx 
		+ \csq\asip(\eIu+\tau \eIu_t,\phih) + \int_{\Gabs} \alpha (\eIu_t + \tau \eIu_{tt})\phih \dG \\+ \intO (f-\fh)\phih \dx + \intGabs (\gabs-\ghabs) \phih \dG+ \intGN (\gN-\ghN) \phih \dG.
	\end{multlined}
\end{eq}
To estimate $\ehu$, we thus need a bound on the term involving the defect; the following uniform bound on $\uht$ will be useful for this purpose.

\begin{lemma} \label{lem: unif bound uht}
	Under the assumptions of Theorem \ref{thm: main}, we have
$
		\|k \uht\|_{L^\infty(0, \thh; \Linf)} \lesssim 1.
$
\end{lemma}\vspace*{-3mm}
\begin{proof}
	By the inverse inequality \eqref{inverse inequality} and the stability of the interpolant in $\Linf$, we have
	\begin{eq}
		\|k \uht \|_{L^\infty(0, \thh; \Linf)} 
		\leq&\, \Linfnormk \left( \|\uht-\Ih \ut \|_{L^\infty(0, \thh; \Linf)} + \|\Ih \ut\|_{L^\infty(0, \thh; \Linf)} \right) \\
		\lesssim&\, \Linfnormk \left(h^{q-d/2} \|\ut \|_{L^\infty(0, \thh; H^{q}(\Omega))}+ \| \ut\|_{L^\infty(0, \thh; \Linf)} \right).
	\end{eq}
Since $\|k\|_{L^\infty(\Omega)} \leq \bark$ is sufficiently small, we can guarantee that $\|k\uht\|_{L^\infty(0,\thh;L^\infty)} \lesssim 1$,  provided $\bark$ is chosen small enough relative to $\|u\|_{\Xu}$.
\end{proof}
We next tackle the estimate of the defect term with the concrete choice of the test function that will be used in the energy analysis, namely $\phih = \eh_{tt}$. This choice is crucial as it yields a bound uniform in both $h$ and $\tau$.
\begin{lemma} \label{lemma: defect}
	Under the assumptions of Theorem~\ref{thm: main}, for any $\eps>0$, we have the following bound:
	\begin{eq}
		\left| \intt \intO \deltah \ehutt \dxs \right|  \lesssim&\, \begin{multlined}[t] h^{2\qp}
			+	 \eps \dgseminorm{\ehut(t)}^2 + \eps \|\ehut(t)\|^2_{\LtwoGabs}\\
			+ \eps \intt \left( \|\ehutt(s)\|^2_{\Ltwo}+\dgseminorm{\ehut(s)}^2 +  \|\ehut(s)\|_{\LtwoGabs}^2 \right)\ds,
		\end{multlined}
	\end{eq}
	where the hidden constant depends on $\|u\|_{\Xu}$ and $T$, but not on $h$ or $\tau$. \vspace*{-3mm}
\end{lemma}
\begin{proof}
	We take $\phih=\eh_{tt}$ in \eqref{defect 1}, integrate the resulting identity over $(0,t)$ and estimate each term on the right-hand side by applying H\"older's inequality, the interpolation estimates of Lemma~\ref{lem: local interpolant properties}, and Young's $\eps$-inequality.  First,
	\begin{eq}
		\tau \inttO \eIu_{ttt} \eh_{tt} \dx \lesssim&\, \bartau \|\eIu_{ttt}\|_{\LtwotLtwo} \|\eh_{tt}\|_{\LtwotLtwo}\\
		\lesssim& \, h^{2q} \|u_{ttt}\|_{L^2(0,T;  H^q(\Omega))} + \eps  \|\eh_{tt}\|_{\LtwotLtwo}^2
	\end{eq}
	for any $\eps>0$. Here the assumption $u \in H^3(0,T; H^q(\Omega))$ is used.
	
	We next estimate the $k$ terms. We have 
	\begin{eq}
		\intt \intO (1+k \uh) \eIu_{tt}  \ehutt \dxs 
		\leq&\, \|1+k \uh\|_{\LinftLinf}\| \eIu_{tt}\|_{\LtwotLtwo}\|\ehutt\|_{\LtwotLtwo}\\
		\lesssim&\,(1+r+r_0)^2 h^{2q}\|\utt\|^2_{L^2(0,T; H^q(\Omega))}+\eps  \|\ehutt\|^2_{\LtwotLtwo} \\
		\lesssim&\,h^{2q}+\eps  \|\ehutt\|^2_{\LtwotLtwo} 	
	\end{eq}
	since $u \in \Xu \hookrightarrow H^2(0,T; H^q(\Omega))$; cf.~\eqref{def Xu}. 
	Next, we estimate
	\begin{eq}
	&  \inttO k \left(\utt \eIu  + \uht \eIu_{t} +\ut \eIu_{t} \right)\ehutt \dxs  \\
	\lesssim&\, \begin{multlined}[t]   \Bigl\{\Linfnormk \|\utt\|_{L^2(0,T; \Linf)}\|\eIu\|_{L^\infty(0,T; \Ltwo)}  +\|k \uht\|_{L^\infty(0,t; \Linf)} \|\eIu_t\|_{\LtwoTLtwo} \\+ \Linfnormk \| \ut\|_{\LinftLinf}\|\eI_t\|_{\LtwoTLtwo}\Bigr\} \|\ehutt\|_{\LtwotLtwo}. 
	\end{multlined}
\end{eq}
	We can rely on the fact that $\|k \uht\|_{L^\infty(0,t; \Linf)}  \lesssim 1$ for $t \in [0, \thh]$ by Lemma~\ref{lem: unif bound uht}. Employing also the fact that $u \in \Xu \hookrightarrow H^2(0,T; \Linf)$ and using the approximation properties of the interpolant, we conclude that
	\begin{eq}
		  \inttO k \left(\utt \eIu  + \uht \eIu_{t} +\ut \eIu_{t} \right)\ehutt \dxs  \lesssim  h^{2q}+\eps  \|\ehutt\|^2_{\LtwotLtwo}.
	\end{eq}
	We next bound the $\ah(\cdot, \cdot)$ term by employing integration by parts in time and then its boundedness stated in \eqref{boundedness dGstar}:
	\begin{eq}
		\intt  \asip( \csq \eIu+\btau \eIu_{t}, \ehutt) \ds =&\, 	  \asip( \csq\eIu(t)+\btau \eIu_{t}(t), \ehut(t)) \ds -	\intt  \asip( \csq \eIu_t+\btau \eIu_{tt}, \ehut) \ds \\
		\lesssim&\,\begin{multlined}[t] 
			\dgstarseminorm{\eIu(t)}^2+ \dgstarseminorm{\eIu_t(t)}^2		 	 + \eps \dgseminorm{\ehut(t)}^2\\+ \intt \dgstarseminorm{\eIu_t(s)}^2 \ds+\intt \dgstarseminorm{\eIu_{tt}(s)}^2 \ds
			+\eps \intt \dgseminorm{\ehut(s)}^2 \ds 
		\end{multlined} \\
		\lesssim&\, h^{2q}	 + \eps \dgseminorm{\ehut(t)}^2+\eps \intt \dgseminorm{\ehut(s)}^2 \ds,
	\end{eq}
	where in the last line we have used the approximation result \eqref{est interpolant dGstar} for the interpolant. 
	To estimate the $\alpha$ term, we employ the continuous trace inequality in Lemma~\ref{lem: cont trace ineq} for $\eIu_t+ \tau \eIu_{tt}$ and the discrete trace inequality \eqref{face element trace ineq} for $\ehutt$:
	\begin{eq}
		&	\intt	\int_{\Gabs} \alpha (\eIu_t+ \tau \eIu_{tt}) \ehutt \dGs\\
		\lesssim&\,\begin{multlined}[t] \sum_{F \in \calFhbndabs}  \intt \left(\|\nabla (\eIu_t+ \tau \eIu_{tt})\|_{\LtwoK}+\hK^{-1}\|  \eIu_t+ \tau \eIu_{tt}\|_{\LtwoK} \right)^{1/2}\\
			\times \|  \eIu_t+ \tau \eIu_{tt}\|^{1/2}_{\LtwoK}\, \hK^{-1/2} \|\ehutt\|_{\LtwoK} \ds.
		\end{multlined}
	\end{eq}
	Then, again by the approximation properties of the interpolant and the fact that $u \in \Xu \hookrightarrow H^2(0,T; H^{q+1}(\Omega))$, we conclude that 
	\begin{eq}
		\intt	\int_{\Gabs} \alpha (\eIu_t+ \tau \eIu_{tt}) \ehut \dGs
		\lesssim&\, h^{2q}+ \eps \|\ehutt\|_{L^2(0,t; \Ltwo)}^2.
	\end{eq}
Using the assumed accuracy of $\fh$ and $\ghabs$ in \eqref{approx properties source terms} and integration by parts in time for the latter term, we infer that
	\begin{eq} \label{est sources}
		&\intt \intO (f-\fh)\ehutt \dxs+   \intt\intGabs (\gabs-\ghabs)\ehutt \dGs \\
		=&\,\intt \intO (f-\fh)\ehutt \dxs +\intGabs (\gabs-\ghabs)(t)\ehut(t) \dG - \intt\intGabs (\gabs-\ghabs)_t \,\ehut \dGs \\
		\lesssim&\, h^{2q}+ \eps \intt \left( \|\ehutt(s)\|_{\Ltwo}^{2}+\|\ehut(s)\|_{\LtwoGabs}^{2}\right) \ds + \eps \|\ehut(t)\|^2_{\LtwoGabs},
	\end{eq}
	where we have also employed the embedding $\gabs-\ghabs \in H^1(0,T; \LtwoG) \hookrightarrow C([0,T]; \LtwoG)$. For the remaining $\gN-\ghN$ term, we proceed similarly by first integrating by parts in time
	\begin{eq} 
		 \intt\intGN (\gN-\ghN)\ehutt \dGs 
		= \intGN (\gN-\ghN)(t)\ehut(t) \dG - \intt\intGN (\gN-\ghN)_t\ehut \dGs, 
	\end{eq}
	but also employ the inverse trace inequality \eqref{inverse inequality}:
		\begin{eq} \label{est ghN}
		&\intGN (\gN-\ghN)(t)\ehut(t) \dG - \intt\intGN (\gN-\ghN)_t\ehut \dGs \\
			\lesssim&\, \begin{multlined}[t]  \sum_{F \in \calFhbndN}  \|\gN(t)-\ghN(t)\|_{L^2(F)} h_K^{-1/2}\|\ehut(t)\|_{\LtwoK}\\[-4mm] \hspace*{1cm}+ \sum_{F \in \calFhbndN} \intt  \|(\gN(s)-\ghN(s))_t\|_{L^2(F)} h_K^{-1/2}\|\ehut(s)\|_{\LtwoK}\ds.
		\end{multlined} 
	\end{eq}
	Then since $\|\gN-\ghN\|_{H^1(0,T; \LtwoGN)} \lesssim h^{q+1}$ by assumption, employing Young's inequality yields
	\begin{eq} 
		 \intt\intGN (\gN-\ghN)\ehutt \dGs 
		\lesssim h^{2q}+ \eps \|\ehut(t)\|^2_{\Ltwo}+ \intt \|\ehut(s)\|_{\Ltwo}^{2} \ds.
	\end{eq}
	Combining all of the above estimates yields the claimed result.
\end{proof}

Equipped with the estimate of the defect term, we are now ready to estimate the discrete error uniformly in $h$ and $\tau$.
\def\Cprop{C_{\textup{Prop}~\ref{prop: a priori est}}}
\begin{proposition} \label{prop: a priori est}
	Under the assumptions of Theorem~\ref{thm: main}, there exists $\Cprop>0$, independent of $h$ and $\tau$, such that
	\begin{eq}
	\tnorm{\ehu(t)}^2 \lesssim h^{2q}, \quad t \in [0, \thh],
	\end{eq}
	where $\tnorm{\cdot}$ is defined in \eqref{def calE}.
\end{proposition}\vspace*{-4mm}
\begin{proof}	
	The proof is unlocked through using the right test function, namely, $\phih=   \ehu_{tt}(s)$. We test  \eqref{error eq}  with $\phih=   \ehu_{tt}(s) \in \Vh$, integrate over $s \in (0,t)$, and integrate by parts in time to arrive at
	\begin{eq}\label{tested1}
		\begin{multlined}[t] 
			\frac{\tau}{2}\| \ehu_{tt}(t)\|^2_{\Ltwo}+  \intt \|\sqrt{1+k \uh}\ehu_{tt}\|^2_{\Ltwo} \ds 
			+\frac{\btau }{2} | \ehu_t(t) |^2_{\dGnorm} \\+ \frac{\alpha}{2} \|\ehu_t(t)\|^2_{\LtwoGabs} + \alpha \tau \intt \|\ehu_{tt}(s)\|^2_{\LtwoGabs}\ds
		\leq   - \csq \intt \ah(\ehu,  \ehu_{tt})\ds\\- \inttO \deltah  \ehu_{tt} \dxs  -  	 \inttO k ((\uht+\ut) \ehu_{t}+ \utt \ehu) \ehu_{tt} \dxs ,
		\end{multlined}
	\end{eq}
	where we have also exploited the coercivity of $\ah(\cdot, \cdot)$ stated in \eqref{coercivity}.   
	
	To estimate the first term on the right-hand side of  \eqref{tested1}, we integrate by parts in time and call upon the boundedness of $\ah(\cdot, \cdot)$:
	\begin{eq}
		\left|	- \csq \intt \ah(\ehu,  \ehu_{tt})\ds \right| =&\, \left|- \csq  \ah(\ehu(t),  \ehu_{t}(t))+ \csq \intt \ah(\ehu_t,  \ehu_{t})\ds \right| \\
		\lesssim&\,	|\ehu(t)|^2_{\dGnorm}+  \eps |\ehu_t(t)|^2_{\dGnorm} +  \intt  |\ehu_t(s)|^2_{\dGnorm} \ds\\
		\lesssim&\,\eps |\ehu_t(t)|^2_{\dGnorm} +  \intt  |\ehu_t(s)|^2_{\dGnorm} \ds,
	\end{eq}
	for any $\eps>0$, where we have also employed the fact that
	\begin{eq}
		|\ehu(t)|^2_{\dGnorm} = 	\left	| \int_0^t \ehu_t(s) \ds \right|^2_{\dGnorm}   \lesssim T \intt  |\ehu_t(s)|^2_{\dGnorm} \ds.
	\end{eq}
	Lemma \ref{lemma: defect} provides us with the estimate of the defect term on the right-hand side of \eqref{tested1}. 
	Next, since $u \in \Xu$ and $\|k \uht\|_{L^2(0,\thh; \Linf)} \lesssim 1$ by Lemma~\ref{lem: unif bound uht}, we have
	\begin{equation}
		- \inttO k ((\uht+\ut) \ehu_{t}+ \utt \ehu) \ehu_{tt} \dxs \lesssim \|\ehut\|^2_{L^2(0,t; \Ltwo) }+ \eps \intt \|\ehutt(s)\|^2_{\Ltwo} \ds,
	\end{equation}
	where we used
		\begin{eq}
		\|\ehu\|^2_{\LtwotLtwo}  \lesssim T ^2 	\|\ehu_t\|^2_{\LtwotLtwo}.
	\end{eq}
	Altogether, starting from \eqref{tested1} and employing these bounds together with Lemma~\ref{lemma: defect}, we obtain
	\begin{eq}\label{tested11}
		&\begin{multlined}[t] 
			\frac{\tau}{2}\| \ehu_{tt}(t)\|^2_{\Ltwo}+ 		\|\ehu_t(t)\|^2_{\Ltwo} +	\intt \|\ehu_{tt}(s)\|^2_{\Ltwo} \ds 
			+\frac{\btau}{2} | \ehu_t(t) |^2_{\dGnorm} \\+ \frac{\alpha  }{2} \|\ehu_t(t)\|^2_{\LtwoGabs}  + \alpha \tau \intt \|\ehu_{tt}(s)\|^2_{\LtwoGabs}\ds
		\end{multlined}
		\\
		\lesssim &\,  \begin{multlined}[t]
			h^{2q}  + \eps \dgseminorm{\ehut(t)}^2		+ \eps \|\ehu_t(t)\|^2_{\LtwoGabs} +  \eps \intt  \|\ehutt(s)\|_{\Ltwo}^2\ds \\
			+\int_{0}^{t}\left( \dgseminorm{\ehut(s)}^2 +\| \ehu_{t}(s)\|^2_{\LtwoGabs}\right) \ds.
		\end{multlined}
	\end{eq}
			Above we have also used the fact that we have  the bound $1+k \uh \geq 1-(r+r_{0})>0$ on $[0, \thh]$, and thus
			\begin{eq}
		\|\ehu_t(t)\|^2_{\Ltwo} +	\intt \|\ehu_{tt}(s)\|^2_{\Ltwo} \ds \lesssim 	\intt \|\ehu_{tt}(s)\|^2_{\Ltwo} \ds  \lesssim	 \intt \|\sqrt{1+k \uh}\ehu_{tt}\|^2_{\Ltwo} \ds .
	\end{eq}
	By then choosing $\eps$ sufficiently small in \eqref{tested11} so that the $\eps$ terms can be absorbed by the ($\tau$-independent) left-hand side terms  (recall that $\btau> b>0$) and then applying \Gronwall's inequality, we obtain
	\begin{eq}\label{tested12}
		&\begin{multlined}[t] 
			\frac{\tau}{2}\| \ehu_{tt}(t)\|^2_{\Ltwo}+  \intt \|\ehu_{tt}(s)\|^2_{\Ltwo} \ds  +	\|\ehu_t(t)\|^2_{\Ltwo} 
			+ | \ehu_t(t) |^2_{\dGnorm} \\+  \|\ehu_t(t)\|^2_{\LtwoGabs}  +  \tau \intt \|\ehu_{tt}(s)\|^2_{\LtwoGabs}\ds
			\lesssim
			h^{2q}.
		\end{multlined}
	\end{eq}
	We also have
	\begin{equation}
		\begin{aligned}
			\|\eh(t)\|^2_{\Ltwo} + |\eh(t)|^2_{\dGnorm} \lesssim T \left(\max_{t \in [0, \thh]}  \|\eh_t(t)\|^2_{ \Ltwo}+ \max_{t \in [0, \thh]} |\eh_t(t)|^2_{\dGnorm} \right) \lesssim h^{2q}.
		\end{aligned}
	\end{equation}
	Adding the above bound to \eqref{tested12} leads to $\|\eh(t)\|^2_{\calE} \lesssim h^{2q}$. We also note that $\Cprop$ is made independent of $\|k\|_{\Linf}$ by exploiting $\|k\|_{\Linf} \lesssim 1$.
\end{proof}

\subsection*{Step III: Extending the existence to the whole time interval} We now have all the necessary ingredients to prove Theorem~\ref{thm: main}.

\begin{proof}[proof of Theorem \ref{thm: main}]
We choose  $\Czero$ in \eqref{eq: def thh}, so that $C_0 > \Cprop$; cf.~Proposition~\ref{prop: a priori est}.  To prove that $(\uh(\thh), \uht(\thh), \uhtt(\thh)) \in \openset$, we note that
	\begin{eq}\label{eq: wellposedness estimate}
		\|k \uh(\thh) \|_{L^{\infty}(\Omega)} 
		\lesssim&  \|k (\uh-\Ih u)(\thh) \|_{L^{\infty}(\Omega)} + \|k \Ih u(\thh) \|_{L^{\infty}(\Omega)} \\
		\lesssim&\, \Linfnormk h^{-d/2}\| (\uh-\Ih u)(\thh) \|_{\Ltwo} + \| k u(\thh)\|_{\Linf} \\
		\lesssim&\, \Linfnormk C_0 h^{q-d/2}+r,
	\end{eq}
on account of Proposition~\ref{prop: a priori est}.  Choosing $\|k\|_{\Linf}$  sufficiently 
small relative to $\Czero$, yields
	\begin{eq}\label{eq: wellposedness estimate3}
		\| k \uh(\thh) \|_{L^{\infty}(\Omega)} < r_{0} + r < 1
	\end{eq}
We thus conclude that $(\uh(\thh),\uht(\thh), \uhtt(\thh)) \in \openset$.   Moreover, by Proposition~\ref{prop: a priori est}, we have
	\[
	\tnorm{\uh(t) - \Ih u(t)}\leq \Cprop h^{q} \leq  \Czero {h^{d/2}} \quad \text{for  }\ t \in [0, \thh]
	\]
	since $\Czero > \Cprop$. \\
	\indent If $\thh <T$, we can apply the local version of the Picard--Lindelöf theorem again at $t=\thh$ to conclude that the solution $\uh$ exists beyond $\thh$. Thus we must have $\thh = T$ and the error bound in Proposition~\ref{prop: a priori est} holds on $[0,T]$.\\
	\indent Since $\uh - u =  \ehu + \Ih u - u$, we can use the approximation properties of the interpolant to conclude that the same error estimate that holds for $\ehu=\uh - \Ih u$ also holds for $ \uh - u$.  This concludes the proof of Theorem~\ref{thm: main}.
\end{proof}
\section{Proof of Theorem~\ref{thm: tau limit}}  \label{Sec: proof 2}
In this section, we prove Theorem~\ref{thm: tau limit} and rigorously establish the connection between $\uh$ and $\uhtauzero$, where the latter solves, for all $\phih \in \Vhq$ and $t \in (0,T]$,
\begin{equation} \label{semi-discrete west} \tag{$P^{\tau = 0}_h$}
	\left\{	\begin{aligned}
		&(((1+k \uhtauzero)\uhttauzero)_t, \phih)_{\Ltwo}+\csq\ah(\uhtauzero, \phih)
		\\& \hspace*{6cm}+ \alpha( \uhttauzero, \phih)_{\LtwoGabs}\\[1mm]
		=&\, (\fhtauzero, \phih)_{\Ltwo} + (\ghabstauzero, \phih)_{\LtwoGabs}+ (\ghNtauzero, \phih)_{\LtwoGN}, \\[1mm] 
		& (\uhtauzero, \uhttauzero)\vert_{t=0}=(\uzeroh,  \uoneh) \in \left(\Vh\right)^2.
	\end{aligned} \right.
\end{equation}  
The following auxiliary result first establishes the well-posedness of  \eqref{semi-discrete west}.%
\begin{proposition}[Well-posedness of the limiting problem] \label{prop: West}
	Let $\csq$, $ b> 0$. Let the assumptions on $\triangles$ made in Section~\ref{sec: dG discretization} hold and let the polynomial degree be $\qp \geq d/2$, where $d \in \{1,2,3\}$. 
	
	Suppose $\utauzero \in H^2(0,T; H^{q+1}(\Omega))$ is the solution of the exact limiting Westervelt problem that satisfies the non-degeneracy condition \eqref{non-degeneracy}.    Let $\fhtauzero \in  C([0,T]; \Vh)$, $\ghabstauzero \in H^1(0,T; \VGhabs)$, and $\ghNtauzero \in H^1(0,T; \VGhN)$ satisfy the accuracy assumptions
	\begin{eq}
		& \|\ftauzero - \fhtauzero\|_{L^{2}(0,T;L^{2}(\Omega))} \lesssim h^{q},  \quad &&
		\|\gabstauzero- \ghabstauzero\|_{H^1(0,T;L^{2}(\Gabs))} \lesssim h^{q},  \\
		& \|\gNtauzero- \ghNtauzero\|_{H^1(0,T;L^{2}(\GN))} \lesssim h^{q+1}. &&
	\end{eq}
	 Let approximate acoustic initial conditions be chosen to interpolate the exact ones,
that is $
		(\uzeroh,  \uoneh) = (\Ih \uzero, \Ih \uone)$. Then there exists $\bark>0$, such that for all $\|k\|_{\Linf} \leq \bark$,  problem \eqref{semi-discrete west} has a unique  solution $\uhtauzero \in C^2([0,T]; \Vhqp)$ satisfying
	\begin{eq} \label{error bound pressure tnorm west}
		\begin{multlined}[t]
			\intt  \|\utauzero_{tt}(s)- \uhtauzero_{tt}(s)\|^2_{\Ltwo}\ds + \|\utauzero_t(t)- \uhtauzero_t(t)\|^2_{\Ltwo} \\
			+| \utauzero_t(t)- \uhtauzero_t(t) |^2_{\dGnorm}+ \| \uttauzero(t)- \uhtauzero_t(t)\|^2_{\LtwoGabs} \\
			+ \|\utauzero(t)- \uhtauzero(t)\|^2_{\Ltwo} 	+| \utauzero(t)- \uhtauzero(t) |^2_{\dGnorm} \lesssim h^{2q}
		\end{multlined}
	\end{eq}
	for $t \in [0,T]$.  Furthermore, there exists $r_0>0$, such that
	\begin{equation}
		1 + k \uhtauzero \geq 1-(r+r_0)>0 \ \text{ a.e. \quad and}  \qquad \|k\uhttauzero\|_{\LinfTLinf} \lesssim 1.
	\end{equation}
\end{proposition} \vspace*{-3mm}
\begin{proof}
	The limiting problem can be analyzed by adapting the arguments of Section~\ref{Sec: proof}, since the estimates established there remain valid for $\tau=0$. The only structural modification is the reformulation of the second-order Westervelt equation as a first-order system with two unknowns. This local existence proof follows analogously to~\cite[Sec.~3.1]{dewit2026}, while the additional Neumann boundary terms are handled exactly as in Section~\ref{Sec: proof}. We therefore omit the details.
\end{proof}
\begin{proof}[proof of Theorem~\ref{thm: tau limit}]
	The difference $\udiff = \uh -\uhtauzero$ satisfies the following time-integrated problem:
	\begin{equation} \label{diff} 
		\begin{aligned}
			&\begin{multlined}[t]  \intT \Bigl \{ (((1+k \uh)\utdiff)_t, \phih)_{\Ltwo}+\ah(\csq\udiff+b \utdiff, \phih)+ \alpha( \utdiff, \phih)_{\LtwoGabs} \\+ ((k \udiff \uhttauzero)_t, \phih)_{\Ltwo} \Bigr\} \ds
			\end{multlined} \\
			=&\, \begin{multlined}[t]-\tau \intT \Big[(\uhttt,\phih)_{\Ltwo}
				+ \csq \ah(\uht,\phih)
				+ \alpha (\uhtt,\phih)_{\LtwoGabs}\Big] \ds \\ \hspace*{1cm}
				+ \intT (\fh-\fhtauzero, \phih)_{\Ltwo} \ds +\intT(\ghabs-\ghabstauzero, \phih)_{\LtwoGabs} \ds
			\end{multlined}
		\end{aligned} 
	\end{equation}
	for all    $\phih \in L^2(0,T; \Vh)$, with $(\udiff, \utdiff)\vert_{t=0}=(0,0)$.
	A canonical testing function for second-order wave problems of this nature would be $\utdiff$.	However, testing with $\utdiff$ would lead to issues with estimating the right-hand side term $-\intt (\tau \uhttt, \utdiff)_{\Ltwo}\ds$.  We could integrate it by parts in time, but \eqref{error bound pressure tnorm} only provides us with a $\tau$ uniform bound on $\sqrt{\tau} \|\uhtt(t)\|_{\Ltwo}$ and not on $\tau  \|\uhtt(t)\|_{\Ltwo}$.  Additionally,  further testing would be needed to get a handle on $\uttdiff$.
	
	To recover the optimal convergence rate in $\tau$, we test instead with a time-nonlocal test function $\phih$ defined as
	\begin{eq}
		\phih(t) = \begin{cases}
			\int_t^{t'} \udiff(s) \ds, &0 \leq t \leq t' \\
			0,  \quad &t' \leq t \leq T.
		\end{cases}
	\end{eq}
	We can then exploit the fact that  $\phih_t = - \udiff$ if $0 \leq t \leq t'$ and $\phih_t =0$ otherwise.  Indeed, after testing \eqref{diff} with $\phi_h$ and integrating by parts in time, we obtain
	\begin{align}
		\int_0^{t'} (((1+k \uh)\utdiff)_t,  \phih)_{\Ltwo}\ds =&\, 		\int_0^{t'} ((1+k \uh)\utdiff,  -\phih_t)_{\Ltwo}\ds\\
		=&\, \int_0^{t'} ((1+k \uh)\utdiff, \udiff)_{\Ltwo}\ds
	\end{align}
	since $\phih(t')= 0$ and $\udiff_t(0) = 0$. 
	We then integrate by parts in time again, which yields
	\begin{align}
		\int_0^{t'} ((1+k \uh)\utdiff, \udiff)_{\Ltwo}\ds
		&=
		\frac12 \bigl((1+k \uh(t'))\udiff(t'), \udiff(t')\bigr)_{\Ltwo} 
		-\frac12 \int_0^{t'} (k \uht\, \udiff, \udiff)_{\Ltwo}\ds.
	\end{align}
	Consequently, we conclude that
	\begin{equation}
		\int_0^{t'} (((1+k \uh)\utdiff)_t, \phih)\ds
		=
		\frac12 \|\sqrt{1+k\uh(t')}\,\udiff(t')\|_{\Ltwo}^2
		-\frac12 \int_0^{t'} (k \uht\,\udiff,\udiff)_{\Ltwo}\ds.
	\end{equation}
	We next treat the $k$ terms. Integrating by parts
\begin{eq}
	- \int_0^{t'}  ((k \udiff \uhttauzero)_t, \phih)_{\Ltwo} \ds =  - \int_0^{t'}  (k \udiff \uhttauzero, \udiff)_{\Ltwo} \ds
\end{eq}
and employing the fact that $\| \uht\|_{\LinfLinf} \lesssim 1$  and $\| \uhttauzero\|_{\LinfLinf} \lesssim 1$, we infer that 
\begin{eq}
	\frac12	\int_0^{t'}	(k \uht \udiff, \udiff)_{\Ltwo} \ds - \int_0^{t'}  (k \udiff \uhttauzero, \udiff)_{\Ltwo} \ds
	\lesssim \int_0^{t'}\|k\|_{\Linf} \|\udiff\|^2_{\Ltwo} \ds;
\end{eq}
the last term will be handled at the end using \Gronwall's inequality. 

Next, using the fact that $\udiff = - \phih_t$ on $(0,t')$ and integrating by parts in time, we obtain
	\begin{eq} \label{interim est} 
		\int_0^{t'}	\ah(\csq\udiff+b \utdiff, \phih) \ds = &\, - \csq \int_0^{t'}	\ah(\phih_t, \phih) \ds+ b \int_0^{t'}	\ah( \utdiff, \phih) \ds \\
		=&\, \frac{\csq}{2} \ah(\phih(0), \phih(0))+b \int_0^{t'}	\ah( \udiff, \udiff) \ds\\
		\gtrsim&\, 	 \frac{\csq}{2} 	\dgseminorm{\phih(0)}^2+b \int_0^{t'} \dgseminorm{\udiff(s)}^2 \ds,
	\end{eq}
	where we have also invoked the coercivity of $\ah(\cdot, \cdot)$.  
	
	We can rewrite the $\alpha$ term on the left-hand side of \eqref{diff}  by integrating by parts in time:
	\begin{eq}
		\alpha \int_0^{t'} ( \utdiff, \phih)_{\LtwoGabs} \ds = \alpha \int_0^{t'} ( \udiff, \udiff)_{\LtwoGabs} \ds.
	\end{eq}
	The $\fh-\fhtauzero$ and $\ghabs-\ghabstauzero$ terms can be handled using H\"older's and Young's inequalities and  assumption  \eqref{tau conv assumption fh ghabs} to obtain
\begin{eq}
	 &\intT (\fh-\fhtauzero, \phih)_{\Ltwo} \ds +\intT(\ghabs-\ghabstauzero, \phih)_{\LtwoGabs} \ds \\
	 \lesssim&\, \begin{multlined}[t]
	 \tau^2
	 				+  \int_0^{t'}\|\phih(s)\|_{\Ltwo}^2\ds +  \eps \int_0^{t'}\|\phih(s)\|_{\LtwoGabs}^2\ds
	 			\end{multlined}
\end{eq}
	for any $\eps>0$.  The last two terms above can be further estimated as follows:
	\begin{eq}
		& \int_0^{t'}\|\phih(s)\|_{\Ltwo}^2\ds +  \eps \int_0^{t'}\|\phih(s)\|_{\LtwoGabs}^2\ds \\
		\lesssim&\, T  \left(\max_{t \in [0, t']} \|\phih\|_{\Ltwo}^2+ \eps\max_{t \in [0,t']} \|\phih\|_{\LtwoGabs}^2 \right) \\
		\lesssim&\, T  \left( \int_0^{t'} \|\udiff(s)\|^2_{\Ltwo} \ds +\eps  \int_0^{t'} \|\udiff(s)\|^2_{\LtwoGabs} \ds\right),
	\end{eq}
	and at the end handled via \Gronwall's inequality or absorbed for small enough $\eps$.
	
It remains to bound the $\tau$ terms on the right-hand side of \eqref{diff}. To handle the $\tau \uhttt$ term, we integrate by parts in time and employ H\"older's and Young's inequalities:
	\begin{eq}
		\int_0^{t'}	 -(\tau \uhttt, \phih)_{\Ltwo} \ds =&\,  (\tau \uhtt(0), \phih(0))_{\Ltwo}+ \int_0^{t'} (\tau \uhtt, \udiff)_{\Ltwo} \ds  \\
		\lesssim&\,\begin{multlined}[t]  \tau^2 \|\utwoh \|^2_{\Ltwo} +\eps  \|\phih(0)\|^2_{\Ltwo}+\tau^2 \|\uhtt\|^2_{\LtwoTLtwo}\\+\int_0^{t'}\|\udiff(s)\|^2_{\Ltwo}\ds.
		\end{multlined}
	\end{eq}
	Here it is crucial that we have the uniform bound $\|\uhtt\|_{\LtwoTLtwo} \leq C$ thanks to Theorem~\ref{thm: main}. Similarly,  
	\begin{eq}
		-\tau \csq  \int_0^{t'}  \ah(\uht, \phih)\ds =&\,  \tau \csq  \ah(\uh(0), \phih(0)) + \tau \csq  \int_0^{t'}  \ah(\uh, \udiff)\ds \\
		\lesssim&\, \tau^2 \dgseminorm{\uzeroh}^2+ \eps \dgseminorm{\phih(0)}^2 +\tau^2 \int_0^{t'} \dgseminorm{\uh(s)}^2\ds+\int_0^{t'} \dgseminorm{\udiff(s)}^2\ds,
	\end{eq}
	where the $\phih(0)$ term can be absorbed by the term resulting from \eqref{interim est} provided $\eps$ is sufficiently small. 
	The final $\tau$ term on the right-hand side of \eqref{diff} can again be tackled via integration by parts in time and H\"older's and Young's inequalities:
	\begin{eq}
		&- \alpha \tau  \int_0^{t'}( \uhtt, \phih)_{\LtwoGabs} \ds \\
		=&\, \alpha \tau (\uht(0), \phih(0))_{\LtwoGabs}- \alpha \tau  \int_0^{t'}( \uht, \udiff)_{\LtwoGabs} \ds \\
		\lesssim&\, \tau^2 \|\uoneh\|^2_{\LtwoGabs}+ \eps \|\phih(0)\|^2_{\LtwoGabs}+ \tau^2 \|\uht\|^2_{L^2(0,T; \LtwoGabs)} + \eps \int_0^{t'} \|\udiff\|^2_{\LtwoGabs} \ds.
	\end{eq}
	The second term on the right can be further bounded as follows: 
	\begin{eq}
		\eps  \|\phih(0)\|^2_{\LtwoGabs} = \eps  \left| \int_0^{t'} \udiff(s) \ds \right|^2_{\LtwoGabs}  \lesssim&\, \eps \left(\int_0^{t'} \|\udiff(s)\|_{\LtwoGabs} \ds \right)^2\\
		  \lesssim&\, \eps T \int_0^{t'} \|\udiff(s)\|^2_{\LtwoGabs} \ds.
	\end{eq} 
	
	Altogether, by employing these estimates in \eqref{diff} and reducing $\eps$ so that the $\eps$ terms can be absorbed, we arrive at 
	\begin{eq} \label{est 2}
		&\|\udiff(t')\|^2_{\Ltwo } + \dgseminorm{\phih(0)}^2+  \int_0^{t'}	\dgseminorm{\udiff(s)}^2 \ds +  \int_0^{t'} \|\udiff(s)\|^2_{\LtwoGabs} \ds \\
		\lesssim&\, \begin{multlined}[t] \tau^2 \left( \dgseminorm{\uzeroh}^2+  \|\uoneh\|^2_{\LtwoGabs} +  \|\utwoh \|^2_{\Ltwo}\right)+\tau^2+ \int_0^{t'} \|\udiff(s)\|^2_{\Ltwo}\ds \\
			+\tau^2\left( \int_0^{t'} \dgseminorm{\uh(s)}^2\ds	+ \int_0^{t'} \|\uht(s)\|^2_{\LtwoGabs}\ds 	+ \|\uhtt\|^2_{\LtwoTLtwo}\right).
		\end{multlined}
	\end{eq}
	The $\tau$-uniform bounds on $ \int_0^{t'} \dgseminorm{\uh(s)}^2\ds$, $\|\uht\|_{L^2(0,T; \LtwoGabs)}$, and $\|\uhtt\|^2_{\LtwoTLtwo}$ from Theorem~\ref{thm: main}, together with an application of \Gronwall's inequality now yield the claim.
\end{proof}
\bibliography{references}{}
  \bibliographystyle{siam} 
\end{document}